\documentclass[11pt,a4paper]{amsart}
\usepackage[utf8]{inputenc}
\usepackage{amsmath,amssymb,amsthm,dsfont,mathtools,wasysym}
\usepackage{color}
\usepackage[active]{srcltx}
\usepackage[top=3.25cm,bottom=2.75cm,left=3.5cm,right=3.5cm]{geometry}
\usepackage[colorlinks=true, linktocpage=true, linkcolor=red!70!black, citecolor=green!50!black, urlcolor=black]{hyperref}
\usepackage{mathrsfs}
\usepackage{needspace}
\usepackage{enumitem}
\usepackage{enumitem}
\setenumerate{label={\rm (\alph{*})}}
\usepackage{trfsigns}
\usepackage{upref}
\usepackage[font=small]{caption}
\usepackage{mathtools}
\usepackage{todonotes}
\usepackage{bbm}
\usepackage{esint}

\usepackage{tikz-cd}
\usetikzlibrary{decorations.pathreplacing}

\makeatletter
\DeclareRobustCommand\widecheck[1]{{\mathpalette\@widecheck{#1}}}
\def\@widecheck#1#2{%
	\setbox\z@\hbox{\m@th$#1#2$}%
	\setbox\tw@\hbox{\m@th$#1%
		\widehat{%
			\vrule\@width\z@\@height\ht\z@
			\vrule\@height\z@\@width\wd\z@}$}%
	\dp\tw@-\ht\z@
	\@tempdima\ht\z@ \advance\@tempdima2\ht\tw@ \divide\@tempdima\thr@@
	\setbox\tw@\hbox{%
		\raise\@tempdima\hbox{\scalebox{1}[-1]{\lower\@tempdima\box
				\tw@}}}%
	{\ooalign{\box\tw@ \cr \box\z@}}}
\makeatother

\DeclareMathOperator{\sgn}{sgn}
\newcommand{\dd}{\,\mathrm{d}}

\newcommand{\ball}{\mathrm{B}}
\newcommand{\bv}{\mathrm{BV}}

\newcommand{\locc}{\mathrm{loc}}
\newcommand{\lebe}{\mathrm{L}}

\newcommand{\dif}{\mathrm{d}}

\newcommand\del\partial
\newcommand\eps\varepsilon
\newcommand\g\gamma
\newcommand\G\Gamma
\newcommand{\hold}{\mathrm{C}}

\newcommand\N{\mathbb{N}}
\newcommand\qq\qquad
\newcommand\R{\mathbb{R}}
\newcommand\spt{{\rm spt}}
\newcommand\vp\varphi
\newcommand\W{{\rm W}}
\newcommand{\xint}[3]{{\setbox0=\hbox{$#1{#2#3}{\int}$}
		\vcenter{\hbox{$#2#3$}}\kern-.5\wd0}}

\newcommand{\sobo}{\W}

\newcommand{\dashint}{\fint}

\newcommand{\homosobo}{{\dot{\sobo}}}
\newcommand{\cube}{\mathrm{Q}}

\numberwithin{equation}{section}

\newtagform{boldbrackets}{\hspace{8ex}\bf(}{\bf)}
\newtagform{standardbrackets}{)}{)}



\allowdisplaybreaks[3]

\newtheorem{theorem}{Theorem}[section]

\newtheorem{lemma}[theorem]{Lemma}

\theoremstyle{remark}
\newtheorem{remark}[theorem]{Remark}

\makeatletter
\newcommand{\setword}[2]{%
  \phantomsection
  #1\def\@currentlabel{\unexpanded{#1}}\label{#2}%
}
\makeatother

\title[Sharp trace spaces for intersections]{Traces of Sobolev functions \\ and higher integrability}
\begin{document}

\author[R. Denk]{Robert Denk}
\author[F. Gmeineder]{Franz Gmeineder}
\author[P. Stephan]{Paul Stephan}

\address{Authors' address: Fachbereich Mathematik und Statistik, Universit\"at Konstanz, Universit\"atsstr.\ 10, 78464 Konstanz, Germany.}
\email{robert.denk@uni-konstanz.de}
\email{franz.gmeineder@uni-konstanz.de}
\email{paul.stephan@uni-konstanz.de}
\date{\today}

\maketitle

\begin{center}
\emph{On the occasion of Hans Triebel's 90th birthday, \\ a giant of function spaces}
\end{center}

\begin{abstract}
We give a sharp characterization of how additional integrability in the interior improves the integrability of boundary traces of $\sobo^{1,p}$-Sobolev functions. The optimality of our results relies on a novel nonlinear extension or lifting operator. 
\end{abstract}

\setcounter{tocdepth}{1}
\tableofcontents

\section{Introduction}
It is a classical fact in the theory of weakly differentiable functions that, on sufficiently regular sets, one may assign boundary values to Sobolev functions. More precisely, let $n\geq 2$ and $1\leq p \leq\infty$. Then, for every open and bounded set $\Omega\subset\R^{n}$ with Lipschitz boundary, there exists a bounded linear and surjective   \emph{trace operator}  $\gamma_{\partial\Omega}\colon\sobo^{1,p}(\Omega)\to Y(\partial\Omega)$ such that $u|_{\partial\Omega}=\gamma_{\partial\Omega}u$ holds for every $u\in\hold(\overline{\Omega})\cap\sobo^{1,p}(\Omega)$. Here, we have set   
\begin{align}\label{eq:tracetarget} 
Y(\partial\Omega) \coloneqq 
\begin{cases}
\lebe^{1}(\partial\Omega)&\;\text{if}\;p=1,\\ 
{\sobo}{^{1-1/p,p}}(\partial\Omega)&\;\text{if}\;1<p<\infty,\\ 
\mathrm{Lip}_{b}(\partial\Omega)&\;\text{if}\;p=\infty.
\end{cases}
\end{align}
For the reader's convenience, the definition of the fractional Sobolev spaces $\sobo^{1-1/p,p}$ is recalled in Section \ref{sec:notation} below. The surjectivity of $\gamma_{\partial\Omega}$ identifies $Y(\partial\Omega)$ as the optimal space to which the boundary values belong, and we thus call $Y(\partial\Omega)$ the \emph{trace space} of $\sobo^{1,p}(\Omega)$. The mapping properties of $\gamma_{\partial\Omega}$ with the spaces $Y(\partial\Omega)$ from \eqref{eq:tracetarget}  are classical, and we refer the reader to, e.g.,  \cite{Leoni,Triebel} for more detail. 

However, refining this setting  in a natural way, much less is known on \emph{the sharp integrability improvements}  when the underlying functions a priori belong to some better $\lebe^{q}(\Omega)$-space than the generic one given by Sobolev's embedding theorem. Hence, in the present paper, we are interested in the basic question of how higher integrability \emph{inside}  $\Omega$ improves the properties of traces of functions \emph{along}  $\partial\Omega$. 

\subsection{Main result} To be precise, let $1\leq p<n$ and $p^{*}<q\leq\infty$, where $p^{*}\coloneqq\frac{np}{n-p}$ is the Sobolev conjugate exponent of $p$. Our objective is to determine the sharp trace space $\gamma_{\partial\Omega}((\sobo^{1,p}\cap\lebe^{q})(\Omega))$; here and in what follows, we abbreviate $(X\cap Y)(\Omega)\coloneqq X(\Omega)\cap Y(\Omega)$ for two function spaces $X(\Omega)$ and $Y(\Omega)$ on $\Omega$. We note that, by the usual Sobolev embedding, $\sobo^{1,p}(\Omega)\hookrightarrow{\lebe}{^{\widetilde{p}}}(\Omega)$ if and only if $\widetilde{p}\leq p^{*}$, and so $(\sobo^{1,p}\cap\lebe^{q})(\Omega)\subsetneq \sobo^{1,p}(\Omega)$ provided that $q>p^{*}$.  Moreover, it is clear that intersecting with $\lebe^{q}$ cannot improve the smoothness of traces beyond $s=1-\frac{1}{p}$. In this regard, we prove that  $\gamma_{\partial\Omega}((\sobo^{1,p}\cap\lebe^{q})(\Omega))$ is precisely the intersection of the classical fractional Sobolev space $\sobo^{1-1/p,p}(\partial\Omega)$ and a Lebesgue space $\lebe^r(\partial \Omega)$ with a specific exponent $r$ depending on $p$ and $q$. In our main result, Theorem \ref{thm:main}, we give a version that also covers the case of homogeneous spaces on half-spaces; we refer the reader to Section \ref{sec:notation} for their precise definition.
\begin{theorem}[Trace space of $\sobo^{1,p}\cap\lebe^{q}$]\label{thm:main}
Let $1 < p < n$ and $p^{*}<q\leq \infty$. Then the following hold: 
\begin{enumerate}
\item\label{item:homogeneous} Denoting the upper half-space by $
\R_{+}^{n}\coloneqq \{(x',x_{n})\colon\;x'\in\R^{n-1},\;x_{n}>0\}$ and identifying $\R^{n-1}\times\{0\}$ with $\R^{n-1}$, we have  
\begin{align}\label{eq:inc1}
    \gamma_{\R^{n-1}}((\homosobo{^{1,p}}\cap\lebe^{q})(\R_{+}^{n}))= ({\homosobo}{^{1-\frac{1}{p}, p}} \cap \lebe^{r})(\R^{n-1}),
\end{align} 
where the exponent $r$ is determined by the scaling relation
\begin{align}\label{eq:rpqconstellation}
r = 1 + q \left( 1 - \frac{1}{p} \right).
\end{align}
In particular, there exists a \emph{continuous extension operator} 
\begin{align*}
\mathrm{E}\colon(\homosobo{^{1-\frac{1}{p},p}}\cap\lebe^{r})(\R^{n-1})\to(\homosobo{^{1,p}}\cap\lebe^{q})(\R_{+}^{n})
\end{align*}
such that $\gamma_{\R^{n-1}}\circ\mathrm{E}=\mathrm{id}$ on $(\homosobo{^{1-1/p,p}}\cap\lebe^{r})(\R^{n-1})$. 
\item\label{item:inhomogeneous} Let $\Omega\subset\R^{n}$ be open and bounded with Lipschitz boundary $\partial\Omega$. Then we have 
\begin{align}\label{eq:inc2}
    \gamma_{\partial\Omega}((\sobo{^{1,p}}\cap\lebe^{q})(\Omega))= ({\sobo}{^{1-\frac{1}{p}, p}} \cap \lebe^{r})(\partial\Omega),
\end{align} 
where the exponent $r$ is determined by the scaling relation \eqref{eq:rpqconstellation}. In particular, there exists a \emph{continuous extension operator} 
\begin{align*}
\mathrm{E}\colon(\sobo{^{1-\frac{1}{p},p}}\cap\lebe^{r})(\partial\Omega)\to(\sobo{^{1,p}}\cap\lebe^{q})(\Omega)
\end{align*}
such that $\gamma_{\partial\Omega}\circ\mathrm{E}=\mathrm{id}$ on $(\sobo{^{1-1/p,p}}\cap\lebe^{r})(\partial\Omega)$. 
\end{enumerate}
In \ref{item:homogeneous} and \ref{item:inhomogeneous}, $\gamma_{\R^{n-1}}$ and $\gamma_{\partial\Omega}$ denote the usual trace operators on $\homosobo^{1,p}(\R_{+}^{n})$ or $\sobo^{1,p}(\Omega)$, respectively. 
\end{theorem}
In view of the limiting exponents, let us note that 
\begin{align*}
q\searrow p^{*}\coloneqq \frac{np}{n-p} \Longrightarrow r \to 1+ \frac{n(p-1)}{n-p} =p\,\frac{n-1}{n-p} \eqqcolon \overline{p}
\end{align*}
and that ${\homosobo}{^{1-1/p,p}}(\R^{n-1})\hookrightarrow \lebe^{\overline{p}}(\R^{n-1})$. As a consequence, the exponents \eqref{eq:rpqconstellation} provide a sharp passage from the $\lebe^{q}$-constrained to the unconstrained case. For future reference we remark that, to the best of our knowledge, Theorem \ref{thm:main} does not follow from routine or even more sophisticated interpolation techniques; see Remark \ref{rem:interpol} for more detail. 

The difficult part of the proof of Theorem \ref{thm:main} is the inclusion '$\supset$' in \eqref{eq:inc1} or \eqref{eq:inc2}, respectively. We have not been able to accomplish this direction by direct use of a linear extension operator. Indeed, the extension or lifting that we introduce in Section \ref{sec:proofmain} to reach the optimal exponent $r$ given by \eqref{eq:rpqconstellation} relies on a nonlinear cut-off procedure. The latter is, a priori, only well-defined for functions which are smooth up to the boundary. By the nonlinearity of our extension, this particularly requires an extra argument to ensure its continuous extendability to the intersection space $(\sobo^{1-1/p,p}\cap\lebe^{r})(\partial\Omega)$. 

The reader might have noticed that Theorem \ref{thm:main} only applies to the exponent range $1<p<n$ but \emph{not} to $p=1$. Indeed, the counterpart of Theorem \ref{thm:main} for $p=1$ is an easier case with a dichotomous statement; the underlying  observation is due to M\"{u}ller \cite{Mueller}, who established the following result.  
\begin{theorem}[$p=1$, M\"{u}ller {\cite{Mueller}}]\label{thm:p=1} Let $p=1$ and $1^{*}\coloneqq \frac{n}{n-1} <q\leq\infty$. Then we have 
\begin{align}\label{eq:p=1halfspace}
\gamma_{\R^{n-1}}((\homosobo{^{1,1}}\cap\lebe^{q})(\R_{+}^{n}))= \begin{cases}
\lebe^{1}(\R^{n-1})&\;\text{if}\;\frac{n}{n-1}<q<\infty, \\ 
(\lebe^{1}\cap\lebe^{\infty})(\R^{n-1})&\;\text{if}\;q=\infty. 
\end{cases}
\end{align}
Likewise, if $\Omega\subset\R^{n}$ is open and bounded with Lipschitz boundary, then we have 
\begin{align}\label{eq:p=1domains}
\gamma_{\partial\Omega}((\sobo{^{1,1}}\cap\lebe^{q})(\Omega))= \begin{cases}
\lebe^{1}(\partial\Omega)&\;\text{if}\;\frac{n}{n-1}<q<\infty, \\ 
\lebe^{\infty}(\partial\Omega)&\;\text{if}\;q=\infty. 
\end{cases}
\end{align} 
Here, $\gamma_{\R^{n-1}}$ and $\gamma_{\partial\Omega}$ denote the usual trace operators on $\homosobo^{1,1}(\R_{+}^{n})$ or $\sobo^{1,1}(\Omega)$, respectively. 
\end{theorem}
In particular, since $\gamma_{\R^{n-1}}(\homosobo{^{1,1}}(\R_{+}^{n}))=\lebe^{1}(\R^{n-1})$ due to a classical result of Gagliardo \cite{GAGLIARDO57} (see also Mironescu \cite{Mironescu15}), the only integrability improvement for traces can be obtained for $q=\infty$; the same applies to domains. Again, the key point is the surjectivity of the trace operator. To keep the paper self-contained, we revisit M\"{u}ller's argument in Section \ref{sec:p=1} but argue closer to Gagliardo's original construction from \cite[\S 3]{GAGLIARDO57}. This comes with technical benefits (e.g., avoiding the detour over $\bv$ for the extensions) and allows for a direct proof of the attainment of the correct traces.

The conclusions of Theorems \ref{thm:main} and   \ref{thm:p=1} remain valid for  operator-adapted spaces of this sort. We refer the reader to  Section \ref{sec:variations} for the precise statements, where we also showcase the need and the construction of a slightly different trace operator as established recently in \cite{BreitDieningGmeineder,DieningGmeineder}; this also provides an independent proof of the inclusion '$\subset$' in \eqref{eq:p=1halfspace} and \eqref{eq:p=1domains} from Theorem \ref{thm:p=1}.  

\subsection{Notation and background material}\label{sec:notation}
We briefly comment on notation. For $n\in\mathbb{N}$, we denote the unit sphere by $\mathbb{S}^{n-1}\coloneqq \{x\in\R^{n}\colon\;|x|=1\}$ and write 
\begin{align*}
\R_{+}^{n}\coloneqq \{x=(x',x_{n})\colon\;x'\in\R^{n-1},\;x_{n}>0\}
\end{align*}
for the upper half-space in the $n$-th coordinate direction; as mentioned above, it is then customary to simply write $\R^{n-1}$ for $\partial\R_{+}^{n}$. As usual, we write $\mathscr{L}^{n}$ and  $\mathscr{H}^{n-1}$ for the $n$-dimensional Lebesgue and the $(n-1)$-dimensional Hausdorff measures; for brevity, we put $\dif x' \coloneqq \dif\mathscr{H}^{n-1}(x')$.  Moreover, for $f\in\lebe^{1}(\R^{n-1})$, the Hardy-Littlewood maximal function is given by 
\begin{align*}
Mf(x') \coloneqq \sup_{r>0}\dashint_{\ball_{r}(x')}|f(z')|\dif z' \coloneqq \sup_{r>0}\frac{1}{|\ball_{r}(x')|}\int_{\ball_{r}(x')}|f(z')|\dif z',\qquad x'\in\R^{n-1}. 
\end{align*}
Here, $|\ball_{r}(x')|\coloneqq \mathscr{H}^{n-1}(\ball_{r}(x'))$. For an open set $U\subset\R^{n}$, $0<s<1$,   $1\leq p<\infty$ and $v\in\lebe_{\locc}^{1}(\partial U)$, we define the $(s,p)$-Gagliardo seminorm by 
\begin{align*}
[v]_{s,p,\partial U}\coloneqq \Big(\iint_{\partial U\times\partial U}\frac{|v(x)-v(y)|^{p}}{|x-y|^{n-1+sp}}\dif\mathscr{H}^{n-1}(x)\dif\mathscr{H}^{n-1}(y)\Big)^{\frac{1}{p}}.
\end{align*}
If $U=\R_{+}^{n}$, we employ the following definition of  homogeneous Sobolev and fractional Sobolev spaces:
\begin{align*}
\homosobo{^{1,p}}(\R_{+}^{n}) &\coloneqq \overline{\hold_{c}^{\infty}(\overline{\R_{+}^{n}})}^{\|\nabla\cdot\|_{\lebe^{p}({\R_{+}^{n}})}}\;\;\;\text{and}\;\;\;\homosobo{^{s,p}}(\R^{n-1}) \coloneqq \overline{\hold_{c}^{\infty}(\R^{n-1})}^{[\cdot]_{s,p,\R^{n-1}}}. 
\end{align*}
A priori, these closures are taken in $\lebe_{\locc}^{1}(\R_{+}^{n})$ and $\lebe_{\locc}^{1}(\R^{n-1})$. For an open set $U\subset\R^{n}$, $\sobo^{s,p}(\partial U)$ is the linear space of all $v\in\lebe^{p}(\partial U)$ such that $[v]_{s,p,\partial U}<\infty$. Lastly, $c,C>0$ denote generic constants that might change from line to line; we only indicate their precise dependencies if they are required in the sequel.

\section{Proof of Theorem \ref{thm:main}}
\subsection{Auxiliary results}
In view of the proof of Theorem \ref{thm:main}, we now record two preparatory lemmas on harmonic extensions and maximal function estimates as follows. The first result is essentially due to Agmon, Douglis and Nirenberg \cite[Theorem 3.3]{Agmon1959}; the version given below is the variant to be found, e.g., in Galdi \cite[Theorem II.11.6]{Galdi}.
\begin{lemma}\label{lem:GaldiLem}
Let the kernel $K$ be of the form
\begin{equation}\label{eq:K}
    K(x',x_n) = \frac{\widetilde{\omega}(x'/|x|, x_n/|x|)}{|x|^{n-1}}, \quad x=(x',x_{n})=(x_1,\ldots,x_{n-1},x_{n}),
\end{equation}
where the function $\widetilde{\omega}$ is continuous on the half-sphere $\{x\in \mathbb{S}^{n-1}: x_n\ge 0\}$
and satisfies a uniform H\"older condition at points $x\in \mathbb{S}^{n-1}$ with $x_n=0$. Suppose that $\partial_i K$ for $i=1,\ldots,n$, and $\partial_n^2 K$ are continuous in $\{x \in \mathbb{R}^n : x_n > 0\}$ and bounded in $\{x \in \mathbb{R}^n : x_n > 0\} \cap \mathbb{S}^{n-1}$ by a positive constant $\kappa$. Moreover, we assume that 
\begin{equation}
    \int_{\{|x'|=1\}}\widetilde{\omega}(x',0)\dd x'=0. \label{int_cond}
\end{equation}
Then, for a function $f \in \lebe^{p}(\R^{n-1})$ with $[f]_{1-1/p,p,\R^{n-1}}<\infty$, 
\begin{equation*}
    u(x',x_n) \coloneqq (K(\cdot,x_n)\ast f)(x') = \int_{\R^{n-1}}K(x'-y',x_n)f(y') \dd y', \quad(x',x_{n})\in\R_{+}^{n}, 
\end{equation*}
belongs to $\lebe^{p}(\mathbb{R}^n_+)$, and the following inequality holds with a constant $c=c(n,p)>0$:
\begin{align}\label{eq:water}
    \|\nabla u\|_{\lebe^{p}(\mathbb{R}^n_{+})}\leq c [ f ]_{1-\frac{1}{p},p, \mathbb{R}^{n-1}}.
\end{align}
\end{lemma}

We verify that the Poisson operator satisfies the conditions of the above theorem. The Poisson kernel for the half-space is given by
\begin{equation}\label{eq:KP}
    K_P(x) \coloneqq c_n \frac{x_n}{|x|^n},\qquad x=(x',x_{n})\in\R_{+}^{n},
\end{equation}
where $c_n >0$ is chosen such that $\int_{\mathbb{R}^{n-1}}K_P(x',x_n) \dd\mathscr{H}^{n-1}(x') = 1$ for all $x_n > 0$. By choosing  $\widetilde{\omega}(\xi) = c_n\xi_n$ with $\xi= x/|x|\in \mathbb{S}^{n-1}$, we find that the kernel is of the form \eqref{eq:K}. Furthermore, \eqref{int_cond} is satisfied, since $\widetilde{\omega}(x', 0) = 0$ holds for all $x' \in \mathbb{R}^{n-1}$. Observing that $K_{P}$ is of class $\hold^\infty$ for $x \neq 0$, we directly conclude that $\partial_i K_{P}$ and $\partial_n^2 K_{P}$ are continuous in $\{ x \in \mathbb{R}^n: x_n > 0 \}$ and bounded in $\{x \in \mathbb{R}^n : x_n > 0\} \cap \mathbb{S}^{n-1}$ for all $i=1,..., n$.  Therefore, we have the following lemma.

\begin{lemma}[Properties of the Poisson extension]\label{lem:harmext}
    Let $1 < p < \infty$, and let $f \in \homosobo^{1-\frac{1}{p},p}(\R^{n-1})$. For $x = (x', x_n) \in \mathbb{R}^{n-1} \times (0,\infty)$, we define the \emph{Poisson extension} by \[
    v(x', x_n)  \coloneqq (K_P(\cdot,x_n) * f)(x'),
    \] 
    where $K_P$ is as in \eqref{eq:KP}. Then $v\in\homosobo^{1,p}(\R_{+}^{n})$, and the $\homosobo^{1,p}$-trace of $v$ satisfies $\gamma_{\R^{n-1}}v = f$ on $\mathbb{R}^{n-1}$.  Moreover, there exists a constant $c=c(n,p)>0$ such that  \[
    \| \nabla v \|_{\lebe^p (\R_{+}^{n})} \leq c\, [ f ]_{1-\frac{1}{p},p,\R^{n-1}}.
    \]
\end{lemma}

\begin{lemma} \label{lem:maximal_estimate}
Let $1 < r < \infty$ and $f \in \lebe^r(\mathbb{R}^{n-1})$. If $v$ denotes the Poisson  extension of $f$ to the upper half-space $\R^{n-1} \times (0,\infty)$ from Lemma \ref{lem:harmext}, then 
\begin{align}\label{eq:cruxmaxest}
|v(x',x_n)| \leq M f(x') \quad \text{for all } (x', x_n) \in \R^{n-1} \times (0, \infty),
\end{align}
where $M$ denotes the Hardy-Littlewood maximal operator on $\R^{n-1}$. Consequently, there exists a constant $C=C(n,r)>0$ such that 
\[
\big\| \sup_{x_n > 0 } |v(\cdot, x_n)| \big\|_{\lebe^r(\mathbb{R}^{n-1})} \leq C \|f\|_{\lebe^r (\R^{n-1})}.
\]
\end{lemma}

\begin{proof}
    The pointwise bound \eqref{eq:cruxmaxest} is shown in \cite[Corollary 2.1.12, Example 2.1.13]{Grafakos2014}. The norm estimate then is a direct consequence of the classical maximal theorem for $r>1$.
\end{proof}

\begin{remark}\label{rem:cutoffprelims}
    Despite being folklore, we explicitly note that the Poisson operator in the unbounded half-space does not map $\homosobo^{1-1/p,p}(\R^{n-1})$ into $\lebe^p(\mathbb{R}_+^n)$. This can be seen as follows:
    Consider the function $f: \mathbb{R}^{n-1} \to \mathbb{R}$ defined by
    $f(x') = (1+|x'|)^{-\alpha}$, where we choose $\alpha$ such that $\frac{n-1}{p} < \alpha \le \frac{n}{p}$. Then $f \in \homosobo^{1-1/p,p}(\mathbb{R}^{n-1})$. Setting $v(x', x_n) \coloneqq (K_P(\cdot,x_n) * f)(x')$, we consider for $x_{n}>1$ the annulus $A_{x_n} = \{ x' \in \mathbb{R}^{n-1} : 1 \leq |x'| < x_n \}$. For $z' \in A_{x_n}$ and $x' \in A_{x_n}$, we have $|x'-z'| < 2x_n$ and so 
\begin{align*}
    K_P(x'-z', x_n) = c_n \frac{x_n}{(x_n^2 + |x'-z'|^2)^{n/2}} \ge c_n \frac{x_n}{(5x_n^2)^{n/2}} = C\, x_n^{-(n-1)}.
\end{align*}
    Moreover, on $A_{x_n}$, we have $f(x') \geq (2x_n)^{-\alpha}$ and so 
    \[
    v(x',x_n) \ge C x_n^{-(n-1)} \int_{A_{x_n}} (2x_n)^{-\alpha} \dd z' = C x_n^{-(n-1)}(x_n^{n-1}-1) x_n^{-\alpha}. 
    \]
    Observing that $x_{n}>2^{\frac{1}{n-1}}$ implies $1-x_{n}^{1-n}>\frac{1}{2}$, we arrive at 
    \begin{align*}
    \| v \|_{\lebe^p(\mathbb{R}_{+}^{n})}^p & \ge \int_{2^{\frac{1}{n-1}}}^{\infty}\int_{A_{x_n}} (C x_n^{-\alpha})^p \dd x' \dif x_{n} \geq C\int_{2^{\frac{1}{n-1}}}^{\infty} x_n^{-\alpha p} (x_{n}^{n-1}-1)\dif x_{n} \\ 
    & \geq C \int_{2^{\frac{1}{n-1}}}^{\infty}x_{n}^{-\alpha p + n-1}(1-x_{n}^{1-n})\dif x_{n} \geq C \int_{2^{\frac{1}{n-1}}}^{\infty}x_{n}^{-\alpha p + n-1}\dif x_{n}, 
    \end{align*}
    where $C=C(n,p,\alpha)>0$ changes from one inequality to the next one. Since we assumed $\alpha \leq \frac{n}{p}$, the remaining integral with respect to $x_n$ diverges. 
\end{remark}

\subsection{Proof of Theorem \ref{thm:main}{\ref{item:homogeneous}} and {\ref{item:inhomogeneous}}}\label{sec:proofmain}
Based on Lemmas \ref{lem:GaldiLem}--\ref{lem:maximal_estimate}, we now come to the proof of Theorem \ref{thm:main}. We first deal with the more difficult case $p^{*}<q<\infty$: 
\begin{proof}[Proof of Theorem \ref{thm:main}, Case  $1<p<p^{*}<q<\infty$]
We consider the homogeneous situation on half-spaces first, see \ref{item:homogeneous}. This is a consequence of a multiplicative trace inequality and an explicit non-linear extension based on the Poisson integral. For notational brevity, we put $X\coloneqq (\homosobo{^{1,p}}\cap\lebe^{q})(\R_{+}^{n})$. 

\textit{Inclusion '$\subset$' in \eqref{eq:inc1}.} The classical trace operator on $\homosobo{^{1,p}}(\R_{+}^{n})$ maps $\gamma_{\R^{n-1}}\colon X\subset\homosobo{^{1,p}}(\R_{+}^{n})\to{\homosobo}{^{1-1/p,p}}(\R^{n-1})$. Hence, it suffices to establish that $\gamma_{\R^{n-1}}\colon X\to\lebe^{r}(\R^{n-1})$. 
This is an easy calculation using H\"older's inequality, which can be found, e.g., in \cite[Chapter~10, Proposition~16.1]{DiBenedetto16}. We include the details here for the sake of completeness. Let $u \in \hold_c^\infty(\overline{\R_{+}^{n}})$. For an arbitrary point $x' \in \R^{n-1}$, the fundamental theorem of calculus yields 
\begin{align}\label{eq:ptwisebound1}
\begin{split}
\lvert u(x', 0)\rvert^r & = - \int_0^\infty \partial_{x_n} \big( \lvert u(x', x_n)\rvert^r \big) \dd x_n \\ &  \leq r \int_0^\infty \lvert u(x', x_n)\rvert^{r-1}  \lvert \nabla u(x', x_n)\rvert \dd x_n.
\end{split}
\end{align}
Because of \eqref{eq:rpqconstellation}, we have $(r-1)p'=q$. Thus, 
integrating \eqref{eq:ptwisebound1} with respect to $x'\in \R^{n-1}$, H\"{o}lder's inequality gives us
\begin{align*}
\int_{\R^{n-1}} \lvert u(x', 0)\rvert^r \dd x' \leq r \left( \int_{\R_{+}^{n}} \lvert \nabla u\rvert^p \dd x \right)^{\frac{1}{p}} \left( \int_{\R_{+}^{n}} \lvert u\rvert^{(r-1)p'} \dd x \right)^{\frac{1}{p'}},
\end{align*}
and so we obtain the multiplicative inequality 
\begin{align*}
\lVert u(\cdot,0)  \rVert_{\lebe^r(\R^{n-1})}^r \leq r \lVert \nabla u \rVert_{\lebe^p(\R_{+}^{n})} \lVert u \rVert_{\lebe^q(\R_{+}^{n})}^{r-1}. 
\end{align*}
By smooth approximation, it then follows that $\gamma_{\R^{n-1}}\colon X\to (\homosobo{^{1-1/p,p}}\cap\lebe^{r})(\R^{n-1})$ boundedly. 

\textit{Inclusion '\,$\supset$' in \eqref{eq:inc1}.} In order to construct an extension operator, we first let $f \in \hold_{c}^{\infty}(\R^{n-1})$ and let $v(x',x_{n})\coloneqq (K_{P}(\cdot,x_{n})*f)(x')$ be the Poisson extension of $f$ to the upper half-space $\R_{+}^{n}$, see Lemma \ref{lem:harmext}. In particular, we have 
\begin{align}\label{eq:stopdraggingmyheartaround}
\gamma_{\R^{n-1}}v = f\;\;\;\text{and}\;\;\; \|\nabla v\|_{\lebe^{p}(\R_{+}^{n})}\leq c\,[f]_{{1-1/p,p,\R^{n-1}}}, 
\end{align}
where $c=c(n,p)>0$. Moreover, by \eqref{eq:cruxmaxest}, we have the pointwise maximal estimate 
\begin{align}\label{eq:pointwisemax}
|v(x', x_n)| \leq (Mf)(x')\qquad\text{for all}\;(x',x_{n})\in{\R_{+}^{n}}.
\end{align}
Since $r>1$, $M$ is bounded on $\lebe^{r}(\R^{n-1})$. 

The lifting $u$ is now defined by truncating $v$ depending on its magnitude. Let $\eta \in \hold^\infty([0, \infty))$ be a cut-off function satisfying $\eta(t)=1$ for $0 \leq t \leq 1$, $\eta(t)=0$ for $t \ge 2$, $0 \leq \eta \leq 1$, and $|\eta'(t)| \leq 2$ for all $t\geq 0$. Define
\begin{align}\label{eq:extdef}
u(x', x_n) \coloneqq   \eta\Big( x_n  |v(x', x_n)|^\beta \Big)v(x', x_n),\qquad (x',x_{n})\in\R_{+}^{n},
\end{align}
where we have put 
\begin{align}\label{eq:pbetar}
\beta \coloneqq \frac{q}{p} - 1\;\;\;\;\text{and therefore}\;\;\;\; q-\beta \stackrel{\eqref{eq:rpqconstellation}}{=} r. 
\end{align}
In view of other potential extensions, this particular choice shall be discussed in Remark \ref{rem:choice} below. Note that, for points $x=(x',x_{n})$ with $u(x)\neq 0$, we have 
\begin{align}\label{eq:xnbeta}
x_n |v(x)|^\beta \leq 2\;\;\;\;\text{and therefore}\;\;\;\;x_n \leq 2 |v(x)|^{-\beta}.
\end{align}
Moreover, since $f\in\hold_{c}^{\infty}(\R^{n-1})$, it follows from \eqref{eq:KP}\emph{ff.} that $\|v\|_{\lebe^{\infty}(\R_{+}^{n})}\leq \|f\|_{\lebe^{\infty}(\R^{n-1})}$. Hence, if $0<x_{n}<\|f\|_{\lebe^{\infty}(\R^{n-1})}^{-1}$, then $x_{n}|v(x',x_{n})|^{\beta}\leq 1$ and so $u(x',x_{n})=v(x',x_{n})$ for all $x'\in\R^{n-1}$ and all such $x_{n}$. Therefore, \eqref{eq:stopdraggingmyheartaround} gives us
\begin{align}\label{eq:itrainsdowninafrica}
\gamma_{\R^{n-1}}u=f
\end{align}
provided that we can establish  $u\in\homosobo{^{1,p}}(\R_{+}^{n})$. This, in particular, shall be a consequence of the next two steps.

\emph{Step 1: $\lebe^{q}$-bounds.} We first give bounds on the $\lebe^q$-norms of $u$. To this end, we use \eqref{eq:pointwisemax}, which holds uniformly in $x_n$. Letting $x'\in\R^{n-1}$ be arbitrary, we split the integral with respect to $x_n$ at $T = 2 (Mf(x'))^{-\beta}$.

For $x_n \leq T$, we estimate $|u(x',x_{n})|$ by $|v(x',x_{n})| \leq Mf(x')$. For $x_n > T$ and assuming that $u(x',x_{n}) \neq 0$, the support condition \eqref{eq:xnbeta} entails the estimate $|u(x',x_{n})| \leq |v(x',x_{n})| \leq (2/x_n)^{1/\beta}$. For $x_n > T$, this bound is strictly smaller than $Mf(x')$; this follows from the definition of $T$. Because of $0\leq\eta\leq 1$, a combination of these estimates yields the pointwise bound 
\begin{align}\label{eq:ptwisebound}
|u(x',x_{n})| \leq G(x',x_{n}) \coloneqq \min\Big\{Mf(x'),\Big(\frac{2}{x_{n}}\Big)^{\frac{1}{\beta}}\Big\},\qquad (x',x_{n})\in\R_{+}^{n}. 
\end{align}
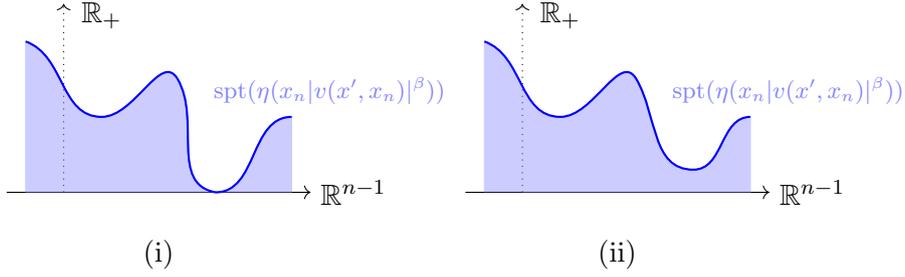
\begin{figure}
\begin{tikzpicture}
\draw[->] (-2.25,0) -- (1.75,0);
\node[below] at (-0.25,-0.5) {(i)};
\node[right] at (1.75,0) {$\R^{n-1}$};
\node[right] at (-1.4,2.35) {$\R_{+}$};
\draw[->,dotted] (-1.5,0) -- (-1.5,2.5);
\draw[-,blue!50!white,fill=blue!50!white,opacity=0.4] (-2,0) --  (-2,2) [out= -20, in =180] to (-1,1) [out=0, in =120] to (0,1.5) [out=-60, in =170] to (0.5,0) [out= 0, in =180] to (1.5,1) --  (1.5,0) -- (-2,0);
\draw[-,blue,thick] (-2,2) [out= -20, in =180] to (-1,1) [out=0, in =120] to (0,1.5) [out=-60, in =170] to (0.5,0) [out=0,in=180] to (1.5,1);
\node[blue!50!white] at (2,1.35) {{\footnotesize $\spt(\eta(x_{n}|v(x',x_{n})|^{\beta}))$}};
\end{tikzpicture}
\begin{tikzpicture}
\draw[->] (-2.25,0) -- (1.75,0);
\node[right] at (1.75,0) {$\R^{n-1}$};
\node[right] at (-1.4,2.35) {$\R_{+}$};
\draw[->,dotted] (-1.5,0) -- (-1.5,2.5);
\draw[-,blue!50!white,fill=blue!50!white,opacity=0.4] (-2,0) -- (-2,2) [out= -20, in =180] to (-1,1) [out=0, in =120] to (0,1.5) [out=-60, in =180] to (0.75,0.3) [out=0,in=180] to (1.5,1) -- (1.5,0) -- (-1.5,0);
\draw[-,blue,thick] (-2,2) [out= -20, in =180] to (-1,1) [out=0, in =120] to (0,1.5) [out=-60, in =180] to (0.75,0.3) [out=0,in=180] to (1.5,1);
\node[blue!50!white] at (2,1.35) {{\footnotesize $\spt(\eta(x_{n}|v(x',x_{n})|^{\beta}))$}};
\node[below] at (-0.25,-0.5) {(ii)};
\end{tikzpicture}
\caption{(Non-)collapse of the support of $u$. If $|v(x',x_{n})|\to\infty$ sufficiently fast as $x_{n}\searrow 0$, then it might happen that $u(x',0)=0$, see (i). If this happens at the points of too large a set,  \eqref{eq:itrainsdowninafrica} might not necessarily be fulfilled. If $v$ is bounded, as is the case in the setting of Step 1 and 2, this scenario cannot occur; see (ii). The conclusion of the continuity argument from Step 3 thus can be interpreted as (i) not happening \emph{too often} for general $(\homosobo^{1-1/p,p}\cap\lebe^{r})$-functions  $f\colon\R^{n-1}\to\R$; see also Remark \ref{rem:choice}.}\label{fig:collapse}
\end{figure}
In consequence, for $(x',x_{n})\in\R_{+}^{n}$, we arrive at 
\begin{align}\label{eq:Lpestimatesstart}
\int_0^\infty |u(x',x_{n})|^q \dd x_n \leq \int_0^T (Mf(x'))^q \dd x_n + \int_T^\infty \left( \frac{2}{x_n} \right)^{\frac{q}{\beta}} \dd x_n.
\end{align}
The first integral equals $T (Mf(x'))^q = 2 (Mf(x'))^{q-\beta}$. The second integral converges since $\frac{q}{\beta} = \frac{pq}{q-p} > 1$ (as $p > 1$), and equals
\begin{align}
\left[ \frac{2^{q/\beta} x_n^{1-q/\beta}}{1-q/\beta} \right]_T^\infty = c\, T^{1-\frac{q}{\beta}} = c\, (Mf(x'))^{-\beta \left(1-\frac{q}{\beta}\right)} = c\, (Mf(x'))^{q-\beta}. \label{eq:secondIntegral}
\end{align}
Since $q-\beta = r$, see \eqref{eq:pbetar}, \eqref{eq:Lpestimatesstart}--\eqref{eq:secondIntegral} yield the  pointwise bound 
\begin{align}\label{eq:tearsforfears}
\int_0^\infty |u(x',x_{n})|^q \dd x_n \leq c\,  (Mf(x'))^r.
\end{align}
Integrating over $x' \in \mathbb{R}^{n-1}$ and applying the boundedness of the maximal operator on $\lebe^r(\mathbb{R}^{n-1})$, we arrive at 
\begin{align}\label{eq:step1concluded}
\lVert u \rVert_{\lebe^q(\R_{+}^{n})}^q \leq c \lVert Mf \rVert_{\lebe^r(\R^{n-1})}^r \leq c \lVert f \rVert_{\lebe^r(\R^{n-1})}^r.
\end{align}

\emph{Step 2: $\lebe^{p}$-gradient bounds.} Next, we give estimates on the weak gradients  $\nabla u$. By the product rule, \begin{align*}
\nabla u(x) = \eta\big( x_n |v(x)|^\beta \big) \nabla v(x) + v(x) \nabla \Big( \eta\big( x_n |v(x)|^\beta \big) \Big) \eqqcolon \mathrm{I}_{1}(x)+\mathrm{I}_{2}(x).
\end{align*}
The term $\mathrm{I}_{1}(x)$ is supported where $\eta(x_{n}|v(x)|^{\beta}) \neq 0$, and is pointwise bounded by $|\nabla v(x)|$. Thus, its $\lebe^p$-norm is controlled by $\lVert \nabla v \rVert_{\lebe^{p}(\R_{+}^{n})}$. For the second term, we define $\Psi(x) \coloneqq x_n |v(x)|^\beta$. In consequence, we have 
\begin{align*}
\nabla \Psi(x) = |v(x)|^\beta \mathbf{e}_n + \beta x_n |v(x)|^{\beta-1} \sgn(v(x)) \nabla v(x),
\end{align*}
where $\mathbf{e}_{n}\coloneqq (0,...,0,1)$. Note that this computation of weak gradients is justified because $v$ is smooth in $\R_{+}^{n}$; the relevant estimates required for the  integrability will be given below. Consequently, the term $\mathrm{I}_{2}(x)$ can be estimated via 
\begin{align*}
|\mathrm{I}_{2}(x)|  & =|v(x)\eta'(\Psi(x)) \nabla \Psi(x)|   \\ & \leq |\eta'(\Psi(x))| \Big( |v(x)|^{\beta+1} + \beta |x_n| |v(x)|^{\beta}|\nabla v(x)| \Big) \eqqcolon A(x) + B(x) 
\end{align*}
with an obvious definition of $A$ and $B$. For term $B$, we observe that on the support of $\eta'$, the argument $\Psi(x',x_{n}) = x_n |v(x',x_{n})|^\beta$ is bounded by 2. Thus, $|B| \leq c |\nabla v|$, and the latter expression belongs to $\lebe^p(\R_{+}^{n})$ by \eqref{eq:stopdraggingmyheartaround}.

For term $A$, we use that $\eta'$ is non-zero only when $1 \leq x_n |v(x',x_{n})|^\beta \leq 2$, which yields the pointwise bound $|v(x', x_n)| \leq (2/x_n)^{1/\beta}$. In addition, the global estimate $|v(x',x_{n})| \leq Mf(x')$ implies that, on the support of $\eta'$, we must have $x_n \ge (Mf(x'))^{-\beta}$. Thus, the integral with respect to $x_n$ is bounded by
\begin{align}
\int_0^\infty |A(x',x_{n})|^p \dd x_n \leq c \int_{(Mf(x'))^{-\beta}}^\infty x_n^{-\frac{p(\beta+1)}{\beta}} \dd x_n, \label{A:is_good}
\end{align}
which converges, since $\frac{p(\beta+1)}{\beta} > 1$. Evaluating at the lower bound yields
\begin{align}\label{eq:sledgehammer}
\left( (Mf(x'))^{-\beta} \right)^{1 - \frac{p(\beta+1)}{\beta}} = (Mf(x'))^{-\beta + p(\beta+1)}.
\end{align}
Recalling the choice $\beta = \frac{q}{p} - 1$ from \eqref{eq:pbetar}, we notice that $p(\beta+1) = q$. Thus, the exponent in \eqref{eq:sledgehammer}  simplifies to
\begin{align}\label{eq:berghammer}
-\beta + p(\beta +1) =- \beta + q \stackrel{\eqref{eq:pbetar}}{=} r. 
\end{align}
Integrating \eqref{A:is_good} with respect to  $x'\in\R^{n-1}$ and using \eqref{eq:sledgehammer}, we arrive at the estimate $\lVert A \rVert_{\lebe^p(\R_{+}^{n})}^p \leq c\, \lVert f \rVert_{\lebe^r(\R^{n-1})}^r$, where $c>0$ is still independent of $f$ and $v$. Summarizing, in conjunction with inequality \eqref{eq:step1concluded} from Step 1, we have established 
\begin{align}\label{eq:chiefbound}
\|u\|_{\lebe^{q}(\R_{+}^{n})}+\|\nabla u\|_{\lebe^{p}(\R_{+}^{n})}\leq c\big(\|f\|_{\lebe^{r}(\R^{n-1})}^{\frac{r}{q}}+\|f\|_{\lebe^{r}(\R^{n-1})}^{\frac{r}{p}}+[f]_{1-1/p,p,\R^{n-1}} \big)
\end{align}
for all $f\in\hold_{c}^{\infty}(\R^{n-1})$. The previous inequality \eqref{eq:chiefbound} entails that $u\in(\homosobo^{1,p}\cap\lebe^{q})(\R_{+}^{n})$, whereby \eqref{eq:itrainsdowninafrica} is now at our disposal.  

\emph{Step 3: Passage to general ${\homosobo}{^{1-1/p,p}}\cap\lebe^{r}$-functions.} It remains to verify that the constructed lifting operator $f\mapsto u$ defined in \eqref{eq:extdef} is continuous from $(\homosobo{^{1-1/p,p}} \cap \lebe^{r})(\R^{n-1})$ to $X$. Let $f,f_{1},f_{2},...\in(\homosobo{^{1-1/p,p}} \cap \lebe^{r})(\R^{n-1})$ be such that $f_{k}\to f$ in $(\homosobo{^{1-1/p,p}} \cap \lebe^{r})(\R^{n-1})$ with respect to the intersection norm. To establish the strong convergence of the extensions $u_k \to u$ in $X$, it suffices to show that every subsequence of $(u_{k})_{k\in\mathbb{N}}$ has a further subsequence converging to $u$ in $X$. Let $(f_{k})_{k\in\mathbb{N}}$ be an arbitrary non-relabelled subsequence.
Since $f_k \to f$ in $\lebe^r(\R^{n-1})$, we can extract another subsequence, still denoted by $(f_k)_{k\in\mathbb{N}}$, such that $f_k \to f$ pointwise $\mathscr{H}^{n-1}$-a.e. on $\R^{n-1}$, and find a majorant $F \in \lebe^r(\R^{n-1})$ such that $|f_k(x')| \leq F(x')$ for $\mathscr{H}^{n-1}$-a.e. $x' \in \R^{n-1}$ and all $k \in \N$.

To see the latter, note that $(f_{k})_{k\in\mathbb{N}}$ is Cauchy with respect to $\|\cdot\|_{\lebe^{r}(\R^{n-1})}$. The classical Riesz-Fischer construction allows us to find a subsequence $(f_{k_{j}})_{j\in\mathbb{N}}$ such that $\|f_{k_{j+1}}-f_{k_{j}}\|_{\lebe^{r}(\R^{n-1})}<2^{-j}$ for all $j\in\mathbb{N}$. It then suffices to put 
\begin{align}\label{eq:RieszFischerMajorant}
F(x') \coloneqq |f_{k_{1}}(x')|+\sum_{l=1}^{\infty}|f_{k_{l+1}}(x')-f_{k_{l}}(x')|,\qquad x'\in\R^{n-1}.
\end{align}
Clearly, $F\in\lebe^{r}(\R^{n-1})$. Moreover, for any arbitrary $j\in\mathbb{N}$, we have $f_{k_{j}}=f_{k_{1}}+\sum_{l=1}^{j-1}(f_{k_{l+1}}-f_{k_{l}})$, and so $|f_{k_{j}}|\leq F$ follows for any $j\in\mathbb{N}$. 

We recall that $v_k$ and $v$ denote the Poisson extensions of $f_k$ and $f$, respectively, and that both $v_{k}$ and $v$ are of class $\hold^{\infty}$ on $\R_{+}^{n}$. Since, in particular, $f_{k}\to f$ strongly in $\homosobo{^{1-1/p,p}}(\R^{n-1})$, it follows from \eqref{eq:water} that $\nabla v_{k}\to \nabla v$ strongly in $\lebe^{p}(\R_{+}^{n};\R^{n})$ and, by Sobolev's embedding, $v_{k}\to v$ strongly in $\lebe^{p^{*}}(\R_{+}^{n})$. As a consequence, there exists another non-relabelled subsequence such that $v_{k}\to v$ and $\nabla v_{k}\to \nabla v$ $\mathscr{L}^{n}$-a.e. in $\R_{+}^{n}$. Going back to \eqref{eq:extdef} and observing that $\eta$ is a continuous function, it follows that $u_k(x) \to u(x) = \eta(x_n |v(x)|^\beta)v(x)$ for $\mathscr{L}^{n}$-a.e. $x \in \R_{+}^{n}$. Furthermore, we recall that 
\begin{align*}
|v_k(x', x_n)| \leq  Mf_k(x') \leq  MF(x') \qquad \text{for all } k \in \N\;\text{and all}\;x'\in\R^{n-1},
\end{align*}
and that $MF \in \lebe^r(\R^{n-1})$. Now let, in analogy to \eqref{eq:ptwisebound}, 
\begin{align*}
 G(x',x_n) \coloneqq \min \Big\{ M F(x'), \left( \frac{2}{x_n} \right)^\frac{1}{\beta} \Big\},\qquad (x',x_{n})\in\R_{+}^{n}, 
 \end{align*}
 so that $|u_k(x',x_n)| \leq |G(x',x_n)|$ pointwise 
 for $\mathscr{L}^{n}$-a.e. $(x',x_{n})\in\R_{+}^{n}$. 
 This follows as in   \eqref{eq:ptwisebound}, where we now work 
 with $T\coloneqq 2(MF(x'))^{-\beta}$, and the fact that $F$ is a majorant for the $f_{k}$'s.  In view of Lebesgue's theorem on dominated convergence, we need to show that $G \in \lebe^q(\R^n_+)$. To this end, we estimate similar to \eqref{eq:Lpestimatesstart}:
 \begin{align*}
\int_0^\infty |G(x', x_n)|^q \dd x_n \leq \int_0^T (MF(x'))^q \dd x_n + \int_T^\infty \left( \frac{2}{x_n} \right)^{\frac{q}{\beta}} \dd x_n.
 \end{align*}
As in \eqref{eq:Lpestimatesstart}--\eqref{eq:step1concluded}, we estimate  
 \begin{align*}
\|G \|_{\lebe^q(\R^n_+)}^{q} \leq c\, \|MF\|_{\lebe^r(\R^{n-1})}^{r} \leq  c\,\|F\|_{\lebe^{r}(\R^{n-1})}^{r}< \infty.
\end{align*}
Therefore, $G \in \lebe^q(\R^n_+)$. Since $u_k \to u$ and $|u_{k}|\leq G$ pointwise $\mathscr{L}^{n}$-a.e. in $\R_{+}^{n}$, dominated convergence implies that  $\|u_k - u\|_{\lebe^q(\R_{+}^{n})} \to 0$.

Next, we consider the weak gradients. Lemma \ref{lem:harmext} entails that, for another non-relabelled subsequence, we have $\nabla v_k \to \nabla v$ pointwise $\mathscr{L}^{n}$-a.e. in $\R_{+}^{n}$ and strongly in $\lebe^{p}(\R_{+}^{n};\R^{n})$. Similar to \eqref{eq:RieszFischerMajorant}, there exists a majorant $H \in \lebe^p(\R_{+}^{n})$ such that $|\nabla v_k| \leq H$ $\mathscr{L}^{n}$-a.e. in $\R_{+}^{n}$. Applying the chain rule yields
\begin{align}\label{long_deri}
\begin{split}
\nabla u_k(x) & = \eta(\Psi_k (x)) \nabla v_k(x)  \\ & + v_k(x) \eta'(\Psi_k (x)) \Big[ |v_k (x)|^\beta \mathbf{e}_n + \beta x_n |v_k (x)|^{\beta-1} \sgn(v_k (x)) \nabla v_k (x) \Big], 
\end{split}
\end{align}
where $\Psi_k (x) \coloneqq x_n |v_k (x)|^\beta$, and likewise with $\Psi (x)\coloneqq x_{n}|v(x)|^{\beta}$:
\begin{align}\label{long_deri_1}
\begin{split}
\nabla u(x) & = \eta(\Psi (x)) \nabla v(x) \\ & + v(x) \eta'(\Psi (x)) \Big[ |v (x)|^\beta \mathbf{e}_n + \beta x_n |v (x)|^{\beta-1} \sgn(v (x)) \nabla v (x) \Big]. 
\end{split}
\end{align}
Again, these computations of weak gradients are justified because $v_{k}$ and $v$ are smooth in $\R_{+}^{n}$; the relevant estimates required for the integrability shall be given below.

Since $v_k \to v$ and $\nabla v_k \to \nabla v$ pointwise $\mathscr{L}^{n}$-a.e. in $\R_{+}^{n}$, the continuity of $\eta$ and $\eta'$ implies  that  $\nabla u_k \to \nabla u$ pointwise $\mathscr{L}^{n}$-a.e. in $\R_{+}^{n}$. To improve the latter to $\lebe^p$-convergence, we construct a dominating function. The first term in \eqref{long_deri} is bounded by $|\nabla v_k| \leq H$. The term involving $\mathbf{e}_n$ is supported where $x_n \geq ( MF(x'))^{-\beta}$ and $|v_k (x)| \leq (2/x_n)^{1/\beta}$, which allows us to estimate
\begin{align*}
\big| v_k(x) \eta'(\Psi_k (x)) |v_k(x)|^\beta \mathbf{e}_n \big| \leq c\, \Big( \frac{2}{x_n} \Big)^{\frac{\beta+1}{\beta}} \mathbbm{1}_{\{ x_n \geq ( MF(x'))^{-\beta} \}}(x)\eqqcolon \widetilde{G}(x). 
\end{align*}
We have $\widetilde{G}\in\lebe^p (\R_{+}^{n})$; this can be seen analogously to  \eqref{A:is_good}\emph{ff.}.
The final term containing $\nabla v_k$ is supported where $x_n |v_k(x)|^\beta \leq 2$, implying boundedness of the coefficient $x_n |v_k(x)|^\beta$. Thus, this term is bounded by $c |\nabla v_k| \leq c H$. 
Combining these estimates, we have $|\nabla u_k| \leq c (H + \widetilde{G}) \in \lebe^p(\R^{n}_+)$ $\mathscr{L}^{n}$-a.e. in $\R_{+}^{n}$. By the dominated convergence theorem, we conclude that $\nabla u_k \to \nabla u$ in $\lebe^p(\R^{n}_+ ;\R^{n})$; in particular, $\nabla u\in\lebe^{p}(\R_{+}^{n};\R^{n})$. Therefore, we have established that $f_{k}\to f$ strongly in $(\homosobo{^{1-1/p,p}}\cap\lebe^{r})(\R^{n-1})$ implies that $u_{k}\to u$ strongly in $(\homosobo{^{1,p}}\cap\lebe^{q})(\R_{+}^{n})$. 
\vspace{0.25cm}

We are now ready to conclude the proof, and so let $f\in(\homosobo{^{1-1/p,p}}\cap\lebe^{r})(\R^{n-1})$ be arbitrary. Since $\hold_{c}^{\infty}(\R^{n-1})$ is dense in $(\homosobo{^{1-1/p,p}}\cap\lebe^{r})(\R^{n-1})$ with respect to the intersection norm, we may choose a sequence $(f_{k})_{k\in\mathbb{N}}\subset\hold_{c}^{\infty}(\R^{n-1})$ such that $f_{k}\to f$ strongly in $(\homosobo{^{1-1/p,p}}\cap\lebe^{r})(\R^{n-1})$. In this situation, \eqref{eq:itrainsdowninafrica} implies that $\gamma_{\R^{n-1}}u_{k}=f_{k}$ for each $k\in\N$, and Step 3 implies that $u_{k}\to u$ strongly in $(\homosobo{^{1,p}}\cap\lebe^{q})(\R_{+}^{n})$. By continuity of the trace operator $\gamma_{\R^{n-1}}\colon\homosobo{^{1,p}}(\R_{+}^{n})\to\homosobo{^{1-1/p,p}}(\R^{n-1})$, it follows that $\gamma_{\R^{n-1}}u_{k}\to\gamma_{\R^{n-1}}u$ strongly in $\homosobo{^{1-1/p,p}}(\R^{n-1})$ and $
\gamma_{\R^{n-1}}u_{k} = f_{k} \to f$ strongly in ${\homosobo}{^{1-1/p,p}}(\R^{n-1})$. 
Hence, $\gamma_{\R^{n-1}}u=f$, and the proof of Theorem \ref{thm:main}\ref{item:homogeneous} for $q<\infty$ 
is complete. 

In view of Theorem \ref{thm:main}\ref{item:inhomogeneous} and Remark \ref{rem:cutoffprelims}, an additional multiplication of \eqref{eq:extdef} with a smooth cut-off in the $n$-th coordinate direction allows to localize the above procedure. In conjunction with routine flattening, this completes the proof.  
\end{proof}
We proceed to address the case $q=\infty$. 
\begin{proof}[Proof of Theorem \ref{thm:main}, case $q=\infty$] Here, we focus on inhomogeneous  spaces, the homogeneous case being analogous. Let  $\eta\in\hold^{\infty}([0,\infty);[0,1])$ be a cut-off function satisfying $\eta(t)=1$ for $t\leq 1$ and $\eta(t)=0$ for $t\geq 2$. Denoting by $\ball_{1}^{(n-1)}(0)$ the open unit ball in $\R^{n-1}$, we introduce for a fixed radially symmetric standard mollifier $\psi\in\hold_{c}^{\infty}(\ball_{1}^{(n-1)}(0))$ and $(x',x_{n})\in\R_{+}^{n}$
\begin{align}\label{eq:extensionLinfty}
    \mathrm{E}f(x', x_n) \coloneqq \eta(x_{n})(\psi_{x_{n}}*u)(x') = \frac{\eta(x_{n})}{x_{n}^{n-1}}\int_{\R^{n-1}}
    \psi\Big(\frac{x'-y'}{x_{n}} \Big)f(y')\dif y', 
    \end{align}
where $f\in\sobo^{1-1/p,p}(\R^{n-1})$. It is well known that $\mathrm{E}\colon\sobo^{1-1/p,p}(\R^{n-1})\to\sobo^{1,p}(\R_{+}^{n})$ is a bounded linear operator and  moreover satisfies  $\gamma_{\R^{n-1}}\mathrm{E}f=f$ $\mathscr{H}^{n-1}$-a.e. on $\R^{n-1}$ for every $f\in\sobo^{1-1/p,p}(\R^{n-1})$; see, e.g., \cite[Corollary 18.29]{Leoni}. Since $\int_{\R^{n-1}}\psi(x')\dif x'=1$, it directly follows from \eqref{eq:extensionLinfty} that $\|\mathrm{E}f\|_{\lebe^{\infty}(\R_{+}^{n})}\leq \|f\|_{\lebe^{\infty}(\R^{n-1})}$. In conclusion, $\mathrm{E}\colon(\sobo^{1-1/p,p}\cap\lebe^{\infty})(\R^{n-1})\to(\sobo^{1,p}\cap\lebe^{\infty})(\R_{+}^{n})$ linearly and boundedly. For an open and bounded set $\Omega\subset\R^{n}$ with Lipschitz boundary, this gives rise to the corresponding lifting operator via localization.

Now let $u\in(\sobo^{1,p}\cap\lebe^{\infty})(\Omega)$. Since $\Omega$ has Lipschitz boundary, we recall from \cite[Theorem 5.7]{EvansGariepy} that, for $\mathscr{H}^{n-1}$-a.e. $x_{0}\in\partial\Omega$, we have 
\begin{align*}
\gamma_{\partial\Omega}u(x_{0}) = \lim_{r\searrow 0}\dashint_{\ball_{r}(x_{0})\cap\Omega}u(x)\dif x \eqqcolon \lim_{r\searrow 0}\frac{1}{\mathscr{L}^{n}(\ball_{r}(x_{0})\cap\Omega)}\int_{\ball_{r}(x_{0})\cap\Omega}u(x)\dif x.  
\end{align*}
Therefore, 
\begin{align}\label{eq:diaz}
|\gamma_{\partial\Omega}u(x_{0})|\leq \limsup_{r\searrow 0}\dashint_{\ball_{r}(x_{0})\cap\Omega}|u(x)|\dif x \leq \|u\|_{\lebe^{\infty}(\Omega)}. 
\end{align}
By arbitrariness of $x_{0}\in\partial\Omega$ and since $\gamma_{\partial\Omega}\colon\sobo^{1,p}(\Omega)\to\sobo^{1-1/p,p}(\partial\Omega)$ is bounded and linear, it follows from \eqref{eq:diaz}  that $\gamma_{\partial\Omega}\colon(\sobo^{1,p}\cap\lebe^{\infty})(\Omega)\to(\sobo^{1-1/p,p}\cap\lebe^{\infty})(\partial\Omega)$ is a bounded linear operator.   This implies the claim. 
\end{proof}

\begin{remark}\label{rem:choice} 
The choice  \eqref{eq:extdef} is designed to get a handle on the behaviour of the extension as $x_{n}\searrow 0$, where the additional $\lebe^{r}(\R^{n-1})$-integrability of the datum particularly must give rise to the improved $\lebe^{q}$-integrability close to $\R^{n-1}$. As a consequence of this construction, if $v(x',x_{n})$ blew up too fast on too large a set $\Sigma\subset\R^{n-1}$ as $x_{n}\searrow 0$, then it could happen that the trace of $u$ from \eqref{eq:extdef} equals zero on $\Sigma$. In this case, $u$ would not have the correct prescribed traces along $\R^{n-1}$; see Figure \ref{fig:collapse}. Together with Steps 1 and 2 in the proof of Theorem \ref{thm:main}\ref{item:homogeneous}, the continuity argument from Step 3 can be understood as ruling out such a behaviour for generic boundary values in $(\homosobo^{1-1/p,p}\cap\lebe^{r})(\R^{n-1})$ on too large a set $\Sigma\subset\R^{n-1}$. 
\end{remark}

\begin{remark}\label{rem:interpol}
Note that the space $(\sobo^{1,p}\cap \lebe^q)(\Omega)$ cannot be seen as a particular case of an anisotropic mixed-norm space $\sobo^{\vec s,\vec p}(\Omega)$ as considered in, e.g., \cite{Johnsen-Sickel08}. For such spaces, the integrability parameter may be different in different directions $x_j$, 
but for each direction it is the same for all derivatives. This is in contrast to the situation considered here. Therefore, trace results on such spaces (see \cite[Theorem~2.2]{Johnsen-Sickel08}, \cite[Proposition~3.7]{Hummel}) cannot be applied.

Moreover, to our knowledge, the results of Theorem~\ref{thm:main} cannot be obtained by a simple interpolation argument. We discuss two possible approaches for a bounded Lipschitz domain $\Omega$.

a) From the classical trace result  and the case $q=\infty$, we obtain the continuity of 
\begin{align*}
    \gamma_{\partial \Omega}\colon (\sobo^{1,p}\cap \lebe^{p^*})(\Omega) &\to (\sobo^{1-1/p,p}\cap \lebe^{\frac{n-1}{n-p}\,p})(\partial\Omega),\\
     \gamma_{\partial \Omega}\colon (\sobo^{1,p}\cap \lebe^{\infty})(\Omega) &\to (\sobo^{1-1/p,p} \cap \lebe^{\infty})(\partial\Omega),
\end{align*}
Here, we use that $\sobo^{1,p}(\Omega) = (\sobo^{1,p}\cap \lebe^{p^*})(\Omega)$ and $\sobo^{1-1/p,p}(\partial\Omega) = (\sobo^{1-1/p,p}\cap \lebe^{\frac{n-1}{n-p}\,p})(\partial\Omega)$ by  Sobolev's embedding theorem 
in $\Omega$ and $\partial\Omega$, respectively. We apply complex interpolation with interpolation parameter  $\theta$ being defined by 
\[ \frac 1q = \frac{1-\theta}{p^*} + \frac{\theta}{\infty}\,. \]
This yields $1-\theta = \frac{p^*}{q} = \frac{np}{(n-p)q}$. By complex interpolation, we obtain the continuity of 
\begin{equation}\label{eq:interp} \gamma_{\partial\Omega}\colon [(\sobo^{1,p}\cap \lebe^{p^*})(\Omega), (\sobo^{1,p} \cap \lebe^{\infty})(\Omega) ]_\theta \to (\sobo^{1-1/p,p} \cap \lebe^{\widetilde r})(\partial\Omega),
\end{equation}
where $\widetilde{r}$ is given by 
\[
    \frac 1{\widetilde{r}} = \frac{1-\theta}{\frac{n-1}{n-p}\,p} = \frac{n}{(n-1)q}\,,
\]
i.e., $\widetilde{r} = \frac{n-1}n\,q$. However, it is well-known that interpolation and intersection do not commute in general. We only have
\[  [(\sobo^{1,p}\cap \lebe^{p^*})(\Omega), (\sobo^{1,p} \cap \lebe^{\infty})(\Omega) ]_\theta \subset (\sobo^{1,p} \cap \lebe^q)(\Omega),\]
but equality does not hold. In fact, a short calculation shows that for $q>p^*$ we get $\widetilde r>r$, so the target space on the right-hand side of \eqref{eq:interp} is smaller than the correct target space in Theorem~\ref{thm:main}. Interpolation and intersection commute only under very strong restrictions on the considered spaces, see, e.g., \cite{Asekritova}, \cite{Garling}. We also refer to the discussion on quasilinearizable interpolation couples and related trace spaces in \cite[Subsection~1.8.4]{Triebel78}.

b) In a second approach, we first interpolate in the interior of the domain and then apply classical trace results. More precisely, we start with the observation that for every $u\in \sobo^{1,p}(\Omega)\cap
\lebe^q(\Omega)$, a fractional Gagliardo-Nirenberg-Sobolev inequality of the form
\begin{equation}\label{eq:interpolation}
     \|u\|_{\sobo^{\theta,p_\theta}(\Omega)} \le C \|u\|_{\sobo^{1,p}(\Omega)}^\theta \|u\|_{\lebe^q(\Omega)}^{1-\theta} 
\end{equation}
holds, where $\theta\in (0,1)$ can be chosen arbitrary and where $p_\theta$ is defined by
\[ \frac1{p_\theta} = \frac\theta p + \frac{1-\theta}q\,.\]
This inequality is a special case of \cite[Theorem~1]{BrezisMironescu}. From \eqref{eq:interpolation}, we get the continuous embedding $(\sobo^{1,p}\cap \lebe^q)(\Omega)\subset \sobo^{\theta,p_\theta}(\Omega)$. If $\theta>1/p_\theta$, we can apply the trace and obtain the continuity of 
\[ \gamma_{\partial\Omega} \colon (\sobo^{1,p}\cap \lebe^q)(\Omega) \to \sobo^{\theta-1/p_\theta, p_\theta}(\partial \Omega).\]
Note that the condition $\theta>1/p_\theta$ is necessary for the continuity of the trace, see, e.g., \cite[Theorem~2.4]{Johnsen-Sickel08}. This condition is equivalent to 
\[ \theta > \frac 1q \Big( 1-\frac 1p+ \frac 1q\Big)^{-1} = \frac{p}{pq-q+p} \eqqcolon \theta_{\min}.\]
As $1/p_\theta = \theta (1/p-1/q)+1/q$ and $q>p$, we see that $p_\theta$ is strictly decreasing as a function of $\theta$. Therefore, by choosing $\theta$ close to $\theta_{\min}$, we can achieve any $p_\theta < p_{\max}$, where 
\[ p_{\max} = \frac{1}{\theta_{\min}} = 1 + q\Big( 1-\frac 1p\Big). \]
Comparing this to \eqref{eq:rpqconstellation}, we see that $p_{\max} = r$. Therefore, this approach yields the continuity of 
\[ \gamma_{\partial\Omega} : (\sobo^{1,p}\cap \lebe^q)(\Omega)\to (\sobo^{1-1/p,p}\cap \lebe^{r-\varepsilon})(\partial\Omega)\]
for any $\varepsilon>0$. This is close to optimal. Note, however, that the case $\varepsilon=0$ cannot be obtained, and that this approach does not give the existence of a continuous extension operator, as we lost information by applying the Gagliardo-Nirenberg-Sobolev inequality \eqref{eq:interpolation}.
\end{remark}
\subsection{Proof of Theorem \ref{thm:p=1}}\label{sec:p=1}
For the sake of completeness and in order to keep our paper self-contained, we now address the case $p=1$ which, in fact, is easier than $1<p<n$. As pointed out in the introduction, the corresponding result is due to M\"{u}ller \cite{Mueller}; our slight modification directly follows Gagliardo's original approach \cite[\S 3]{GAGLIARDO57} and gives rigorous justification of why the underlying extension attains the correct traces indeed. 
\begin{proof}[Proof of Theorem \ref{thm:p=1}]
By routine flattening and localization techniques, it again suffices to cover the homogeneous situation on half-spaces. 

\textit{Inclusion `$\subset$' in \eqref{eq:p=1halfspace}.} If $\frac{n}{n-1}<q<\infty$, then this inclusion directly follows from $(\homosobo^{1,1}\cap\lebe^{q})(\R_{+}^{n})\subset\homosobo^{1,1}(\R_{+}^{n})$ and the fact that the usual trace operator on $\homosobo^{1,1}(\R_{+}^{n})$ maps $\homosobo^{1,1}(\R_{+}^{n})$ to $\lebe^{1}(\R^{n-1})$. If $q=\infty$, then the same argument as in the situation on domains (see, e.g., \cite[Theorem 5.7]{EvansGariepy}) yields that 
\begin{align}\label{eq:debramorgan}
\gamma_{\R^{n-1}}u(x)=\lim_{r\searrow 0}\dashint_{\ball_{r}(x)\cap\R_{+}^{n}}u(y)\dif y
\end{align}
holds for all $u\in\homosobo^{1,1}(\R^{n}_+)$ and $\mathscr{H}^{n-1}$-a.e. $x\in\R^{n-1}$. From here, we deduce that $\|\gamma_{\R^{n-1}}u\|_{\lebe^{\infty}(\R^{n-1})}\leq \|u\|_{\lebe^{\infty}(\R_{+}^{n})}$ provided that $u\in(\homosobo^{1,1}\cap\lebe^{\infty})(\R_{+}^{n})$, and so we infer that $\gamma_{\R^{n-1}}\colon(\homosobo^{1,1}\cap\lebe^{\infty})(\R_{+}^{n})\to(\lebe^{1}\cap\lebe^{\infty})(\R^{n-1})$ too.

\textit{Inclusion '\,$\supset$' in \eqref{eq:p=1halfspace}.} We firstly let $\frac{n}{n-1}<q<\infty$ and let $f\in\lebe^{1}(\R^{n-1})$ be arbitrary. We choose sequences $(f_{j})_{j\in\mathbb{N}}\subset\hold_{c}^{\infty}(\R^{n-1})$ such that 
\begin{align}\label{eq:giveintome}
\begin{split}
    &\|f-f_{j}\|_{\lebe^{1}(\R^{n-1})}\leq 2^{-j}\|f\|_{\lebe^{1}(\R^{n-1})}\qquad\text{for all}\;j\in\mathbb{N}, 
    \end{split}
\end{align}
and note that \eqref{eq:giveintome} holds true for $j=0$ and $f_{0}\coloneqq 0$. We then define, with $\nabla'$ denoting the tangential gradient along $\R^{n-1}$, for $j\in\mathbb{N}_{0}$: 
\begin{align}\label{eq:gammachoose}
\begin{split}
& \gamma_{j}\coloneqq \|f_{j}\|_{\lebe^{q}(\R^{n-1})}^{q} + \|\nabla'f_{j}\|_{\lebe^{1}(\R^{n-1})}, \\ & t_{j+1} \coloneqq t_{j} - s _{j} \coloneqq t_{j} - 2^{-j-1}\frac{\|f\|_{\lebe^{1}(\R^{n-1})}^{q}}{1+\gamma_{j}+\gamma_{j+1}+\|f\|_{\lebe^{1}(\R^{n-1})}^{q}}\;\;\;\text{and}\;\;\; t_{0}\coloneqq \sum_{j=0}^{\infty}s_{j}. 
\end{split}
\end{align}
We follow Gagliardo's classical construction (see, e.g., \cite[\S 3]{GAGLIARDO57}) and define $u\colon\R_{+}^{n}\to\R$ by 
\begin{align}\label{eq:thriller}
u(x',x_{n}) \coloneqq \frac{t_{j}-x_{n}}{t_{j}-t_{j+1}}f_{j+1}(x') + \frac{x_{n}-t_{j+1}}{t_{j}-t_{j+1}}f_{j}(x'),\qquad t_{j+1}\leq x_{n} \leq t_{j}
\end{align}
with $j\in\mathbb{N}_{0}$, and $u(x',x_{n})\coloneq 0$ if $x_{n}\geq t_{0}$. This function is continuous on $\R_{+}^{n}$ and satisfies, for all $\ell\in\{1,...,n-1\}$ and all $j\in\mathbb{N}_0$,  
\begin{align}
&|u(x',x_{n})| \leq |f_{j}(x')|+|f_{j+1}(x')|, \notag \\ 
&|\partial_{\ell}u(x',x_{n})| \leq |\nabla'f_{j+1}(x')| + |\nabla'f_{j}(x')|, \label{eq:gradbounds1}\\ & |\partial_{n}u(x',x_{n})| \leq \frac{|f_{j+1}(x')-f_{j}(x')|}{t_{j}-t_{j+1}}\notag
\end{align}
whenever $x'\in\R^{n-1}$ and $t_{j+1}\leq x_{n} \leq t_{j}$ with $j\in\mathbb{N}_{0}$. We estimate  
\[
\int_{\R_{+}^{n}} |u(x)|^{q}\,\dif x  \stackrel{\eqref{eq:gradbounds1}_{1}}{\leq}  c \sum_{j=0}^{\infty} (t_{j}-t_{j+1})\int_{\R^{n-1}}|f_{j}(x')|^{q}+|f_{j+1}(x')|^{q} \,\dif x'\stackrel{\eqref{eq:gammachoose}}{\leq} c\,\|f\|_{\lebe^{1}(\R^{n-1})}^{q}, 
\]
where $c=c(n,q)>0$ is a constant. Moreover, we have for all $\ell\in\{1,...,n-1\}$
\begin{align*}
\int_{\R^{n-1}\times (0,t_{0})}|\partial_{\ell}u(x',x_{n})|\,\dif x & \stackrel{\eqref{eq:gradbounds1}_{2}}{\leq} \sum_{j=0}^{\infty}(t_{j}-t_{j+1})(\|\nabla'f_{j+1}\|_{\lebe^{1}(\R^{n-1})}+\|\nabla'f_{j}\|_{\lebe^{1}(\R^{n-1})}) \\ & \;\stackrel{\eqref{eq:gammachoose}}{\leq} c\,\|f\|_{\lebe^{1}(\R^{n-1})}^{q}
\end{align*}
and 
\begin{align}\label{eq:spieler}
\begin{split}
\int_{\R^{n-1}\times (0,t_{0})}|\partial_{n}u(x',x_{n})|\,\dif x & \stackrel{\eqref{eq:gradbounds1}_{3}}{\leq}  \sum_{j=0}^{\infty}\int_{\R^{n-1}}|f_{j+1}(x')-f_{j}(x')|\dif x' \\ & \;\stackrel{\eqref{eq:giveintome}}{\leq} c\,\|f\|_{\lebe^{1}(\R^{n-1})}. 
\end{split}
\end{align}
The preceding three estimates imply that $u\in(\dot{\sobo}{^{1,1}}\cap\lebe^{q})(\R_{+}^{n})$, and we note that $u$ is continuous in $\R_{+}^{n}$ by construction. We finish the proof by establishing that $u$ indeed has the correct $\homosobo^{1,1}$-trace. To this end, we recall that the $\homosobo^{1,1}$-trace $\gamma_{\R^{n-1}}u$ is defined as the uniquely determined $\lebe^{1}(\R^{n-1})$-limit of $(u_{k}(x',0))_{k\in\mathbb{N}}$, where $(u_{k})_{k\in\mathbb{N}}\subset\hold_{c}^{\infty}(\overline{\R_{+}^{n}})$ is such that $\|\nabla u-\nabla u_{k}\|_{\lebe^{1}(\R_{+}^{n})}\to 0$ as $k\to\infty$. In particular, the trace does not depend on the specific choice of the sequence  $(u_{k})_{k\in\mathbb{N}}$ with the aforementioned properties. 

In the situation considered here, the membership of the continuous function $u$ in $({\dot{\sobo}}{^{1,1}}\cap\lebe^{q})(\R_{+}^{n})$ allows us to find a sequence $(u_{k})_{k\in\mathbb{N}}\subset\hold_{c}^{\infty}(\overline{\R_{+}^{n}})$ such that $u_{k}\to u$ strongly in $\lebe^{q}(\R_{+}^{n})$, $u_{k}\to u$ pointwise in $\R_{+}^{n}$ and $\nabla u_{k}\to\nabla u$ strongly in $\lebe^{1}(\R_{+}^{n};\R^{n})$. By passing to a non-relabelled subsequence, it is no loss of generality to assume that $u_{k}(x',0)\to \gamma_{\R^{n-1}}u(x')$ for $\mathscr{H}^{n-1}$-a.e. $x'\in \R^{n-1}$. Now, for $x_{n}>0$, Fatou's lemma and the fundamental theorem of calculus then combine to 
\begin{align}\label{eq:JanetJackson}
\begin{split}
\int_{\R^{n-1}} |u(x',x_{n})& -\gamma_{\R^{n-1}}u(x')|\, \dif  x'  \\ & \leq \liminf_{k\to\infty} \int_{\R^{n-1}}|u_{k}(x',x_{n})-u_{k}(x',0)|\, \dif x' \\ 
& \leq \liminf_{k\to\infty} \int_{0}^{x_{n}}\int_{\R^{n-1}}|\nabla u_{k}(x',t)|\,\dif x'\dif t \\ 
& = \int_{0}^{x_{n}}\int_{\R^{n-1}}|\nabla u(x',t)|\,\dif x'\dif t, 
\end{split}
\end{align}
where the ultimate line follows from the strong convergence $\nabla u_{k}\to\nabla u$ in $\lebe^{1}(\R_{+}^{n};\R^{n})$; note that, since $u$ is continuous in $\R_{+}^{n}$, the integral in $\eqref{eq:JanetJackson}_{1}$ is well-defined for any $x_{n}>0$. The estimate \eqref{eq:JanetJackson} entails that there exists a sequence $(x_{n}^{(i)})_{i\in\mathbb{N}}\subset(0,t_{0})$ with $x_{n}^{(i)}\searrow 0$ such that 
\begin{align}\label{eq:MichaelJackson}
\lim_{i\to\infty} \int_{\R^{n-1}} |u(x',x_{n}^{(i)})& -\gamma_{\R^{n-1}}u(x')|\, \dif  x' = 0.
\end{align}
For every $i\in\mathbb{N}$, there exists $j = j(i)\in\mathbb{N}_{0}$ with $t_{j+1}\leq x_{n}^{(i)} \leq t_{j}$, and we note that $i\to\infty$ implies that $j=j(i)\to\infty$. Based on \eqref{eq:thriller}, we record that 
\begin{align}
|u(x',x_{n}^{(i)})-f(x')| & \leq \left\vert \frac{t_{j}-x_{n}}{t_{j}-t_{j+1}}(f_{j+1}(x')-f(x')) + \frac{x_{n}-t_{j+1}}{t_{j}-t_{j+1}}(f_{j}(x')-f(x'))\right\vert \notag \\ 
& \leq |f_{j+1}(x')-f(x')| + |f_{j}(x')-f(x')|. 
\label{eq:dangerous}
\end{align}
In view of \eqref{eq:giveintome}, integrating the previous inequality over $x'\in\R^{n-1}$ gives us 
\begin{align*}
\int_{\R^{n-1}}|f(x')-\gamma_{\R^{n-1}}u(x')|\, \dif x' & \leq \int_{\R^{n-1}}|u(x',x_{n}^{(i)})-\gamma_{\R^{n-1}}u(x')|\, \dif x' \\ 
& \!\!\!\!\!\!\!\!\!\! + \int_{\R^{n-1}}|u(x',x_{n}^{(i)})-f(x')|\, \dif x' \\ 
& \!\!\!\!\!\!\!\!\!\!\!\!\!\!\!\!\!\!\!\!\!\! \stackrel{\eqref{eq:giveintome},\,\eqref{eq:dangerous}}{\leq} \int_{\R^{n-1}}|u(x',x_{n}^{(i)})-\gamma_{\R^{n-1}}u(x')|\, \dif x'\\ 
& \!\!\!\!\!\!\!\!\!\! + c\,2^{-j}\|f\|_{\lebe^{1}(\R^{n-1})}\\ 
& \!\!\!\!\!\!\!\!\!\!\!\!\!\!\!\!\!\!\!\!\!\!\! \stackrel{i,j\to\infty,\,\eqref{eq:MichaelJackson}}{\longrightarrow} 0.
\end{align*}
In conclusion, we have  $\gamma_{\R^{n-1}}u=f$ $\mathscr{H}^{n-1}$
-a.e. in $\R^{n-1}$, and this completes the proof in the case of half-spaces for $\frac{n}{n-1}<q<\infty$.

If $q=\infty$ and accordingly $f\in(\lebe^{1}\cap\lebe^{\infty})(\R^{n-1})$, we not only have \eqref{eq:giveintome} for a suitable sequence $(f_{j})_{j\in\mathbb{N}}$, but can also achieve that  
\begin{align}\label{eq:unifLinftybounds}
\sup_{j\in\mathbb{N}_{0}}\|f_{j}\|_{\lebe^{\infty}(\R^{n-1})}\leq \|f\|_{\lebe^{\infty}(\R^{n-1})}. 
\end{align}
As a substitute of \eqref{eq:gammachoose}, we now put 
\begin{align*}
& \gamma_{j}\coloneqq \|f_{j}\|_{\lebe^{\infty}(\R^{n-1})} + \|\nabla'f_{j}\|_{\lebe^{1}(\R^{n-1})}, \\ & t_{j+1} \coloneqq t_{j} - s _{j} \coloneqq t_{j} - 2^{-j-1}\frac{\|f\|_{\lebe^{1}(\R^{n-1})}}{1+\gamma_{j}+\gamma_{j+1}+\|f\|_{\lebe^{1}(\R^{n-1})}}\;\;\;\text{and}\;\;\; t_{0}\coloneqq \sum_{j=0}^{\infty}s_{j}. 
\end{align*}
Introducing $u$ as in \eqref{eq:thriller}, we may argue as in the case $q<\infty$ to deduce that $u\in\homosobo^{1,1}(\R_{+}^{n})$ and $\gamma_{\R^{n-1}}u=f$. By \eqref{eq:unifLinftybounds}, it immediately follows from \eqref{eq:thriller} that $u\in\lebe^{\infty}(\R_{+}^{n})$ too, and this implies the claim. 
\end{proof}
\begin{remark}[Consistency of Theorems \ref{thm:main} and \ref{thm:p=1}]
Sending $p\searrow 1$ in \eqref{eq:rpqconstellation} yields $r\searrow 1$ provided that $q<\infty$. Thus, if we \emph{informally} interpret $\homosobo^{0,1}$ as $\lebe^{1}$, $\eqref{eq:p=1halfspace}_{1}$ confirms the correct limiting case for $p\searrow 1$. 
\end{remark}
\section{Variations of the theme}\label{sec:variations}
We conclude the paper by addressing variants of Theorems \ref{thm:main} and \ref{thm:p=1} that involve spaces adapted to differential operators. This was indeed our original motivation, see also Remark \ref{rem:generalizations}, and so we briefly describe the underlying context. To this end, let $\mathbb{A}=\sum_{j=1}^{n}A_{j}\partial_{j}$ be a vectorial first order differential operator where, for each $j\in\{1,...,n\}$, $A_{j}\colon\R^{N}\to\R^{M}$ is a fixed linear map. Following \cite{BreitDieningGmeineder,Kalamajska,Smith}, we say that $\mathbb{A}$ is $\mathbb{C}$\emph{-elliptic} provided that its Fourier symbol  $\mathbb{A}[\xi]\coloneqq \sum_{j=1}^{n}A_{j}\xi_{j}\colon \mathbb{C}^{N}\to\mathbb{C}^{M}$ is injective for each $\xi=(\xi_{1},...,\xi_{n})\in\mathbb{C}^{n}\setminus\{0\}$. For an open and bounded set $\Omega\subset\R^{n}$ with Lipschitz boundary, we define for $1\leq p<\infty$  
\begin{align}\label{eq:BVA}
\mathrm{W}^{\mathbb{A},p}(\Omega)\coloneqq \{u\in\lebe^{1}(\Omega;\R^{N})\colon\;\mathbb{A}u\in\lebe^{p}(\Omega;\R^{M})\}. 
\end{align}
For a half-space $\mathbb{H}\subset\R^{n}$, the homogeneous variants $\homosobo^{\mathbb{A},p}(\mathbb{H})$ are defined as the $\|\mathbb{A}\cdot\|_{\lebe^{p}(\mathbb{H})}$-closures of $\hold_{c}^{\infty}(\overline{\mathbb{H}};\R^{N})$ in $\lebe_{\locc}^{1}(\mathbb{H};\R^{N})$. Subject to the $\mathbb{C}$-ellipticity of $\mathbb{A}$ and assuming $1<p<\infty$, it follows from \cite{DieningGmeineder,Kalamajska,Smith} that $\sobo^{\mathbb{A},p}(\Omega)=\sobo^{1,p}(\Omega;\R^{N})$ and $\homosobo^{\mathbb{A},p}(\mathbb{H})=\homosobo^{1,p}(\mathbb{H};\R^{N})$. Hence,  Theorem \ref{thm:main} applies. 

In general, one has $\sobo^{1,1}(\Omega;\R^{N})\subsetneq\sobo^{\mathbb{A},1}(\Omega)$ (see, e.g., \cite{Ornstein}), a fact that is intimately linked with the unboundedness of Calder\'{o}n-Zygmund operators on $\lebe^{1}$. Following \cite{BreitDieningGmeineder,DieningGmeineder,GmeinederRaitaVanSchaftingen2021,GmeinederRaitaVanSchaftingen2024}, the $\mathbb{C}$-ellipticity of $\mathbb{A}$ however allows to conclude the existence of a surjective trace operator $\gamma_{\partial\Omega}\colon\sobo^{\mathbb{A},1}(\Omega)\to\lebe^{1}(\partial\Omega;\R^{N})$, with an analogous result for the homogeneous variants on half-spaces \emph{irrespectively} of the particular half-space. Therefore, in the following, we may restrict ourselves to the particular choice $\mathbb{H}=\R_{+}^{n}$. Here, we have the following result:  
\begin{theorem}\label{thm:operatoradapted}
Let $\mathbb{A}$ be a $\mathbb{C}$-elliptic differential operator of the above form. Moreover, let $p=1$ and $1^{*}\coloneqq \frac{n}{n-1} <q\leq\infty$. Then we have 
\begin{align}\label{eq:p=1halfspaceA}
\gamma_{\R^{n-1}}((\homosobo{^{\mathbb{A},1}}\cap\lebe^{q})(\R_{+}^{n}))= \begin{cases}
\lebe^{1}(\R^{n-1})&\;\text{if}\;\frac{n}{n-1}<q<\infty, \\ 
(\lebe^{1}\cap\lebe^{\infty})(\R^{n-1})&\;\text{if}\;q=\infty. 
\end{cases}
\end{align}
Likewise, if $\Omega\subset\R^{n}$ is open and bounded with Lipschitz boundary, then we have 
\begin{align}\label{eq:p=1domainsA}
\gamma_{\partial\Omega}((\sobo{^{\mathbb{A},1}}\cap\lebe^{q})(\Omega))= \begin{cases}
\lebe^{1}(\partial\Omega)&\;\text{if}\;\frac{n}{n-1}<q<\infty, \\ 
\lebe^{\infty}(\partial\Omega)&\;\text{if}\;q=\infty. 
\end{cases}
\end{align} 
Here, $\gamma_{\R^{n-1}}$ and $\gamma_{\partial\Omega}$ denote the trace operators on $\homosobo^{\mathbb{A},1}(\R_{+}^{n})$ or $\sobo^{\mathbb{A},1}(\Omega)$, respectively. 
\end{theorem}

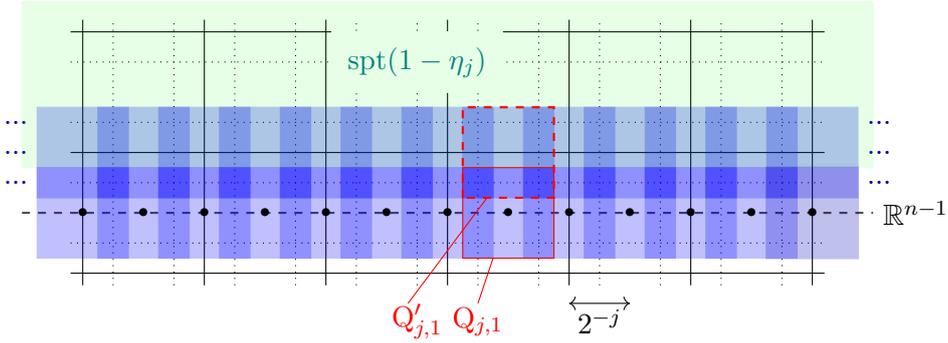
\begin{figure}[t]
	\begin{center}
	\begin{tikzpicture}[scale=0.8]
		\draw[-] (-6.2,1) -- (6.2,1);
		\draw[dashed, thick,opacity=0.7] (-7,-2) -- (7,-2);
		\draw[-] (-6.2,-1) -- (6.2,-1);
		\draw[-] (-6.2,-3) -- (6.2,-3);
		\draw[-] (-2,-3.2) -- (-2,1.2);
		\draw[-] (-4,-3.2) -- (-4,1.2);
		\draw[-] (0,-3.2) -- (0,1.2);
		\draw[-] (2,-3.2) -- (2,1.2);
		\draw[-] (4,-3.2) -- (4,1.2);
		\draw[-] (6,-3.2) -- (6,1.2);
		\draw[-] (-6,-3.2) -- (-6,1.2);
		\draw[dotted] (-6.2,0.5) -- (6.2,0.5);
		\draw[dotted] (-6.2,-0.5) -- (6.2,-0.5);
		\draw[dotted] (-6.2,-1.5) -- (6.2,-1.5);
		\draw[dotted] (-6.2,-2.5) -- (6.2,-2.5);
		\draw[dotted] (-5.5,-3.2) -- (-5.5,1.2);
		\draw[dotted] (-4.5,-3.2) -- (-4.5,1.2);
		\draw[dotted] (-3.5,-3.2) -- (-3.5,1.2);
		\draw[dotted] (-2.5,-3.2) -- (-2.5,1.2);
		\draw[dotted] (-1.5,-3.2) -- (-1.5,1.2);
		\draw[dotted] (-0.5,-3.2) -- (-0.5,1.2);
		\draw[dotted] (0.5,-3.2) -- (0.5,1.2);
		\draw[dotted] (1.5,-3.2) -- (1.5,1.2);
		\draw[dotted] (2.5,-3.2) -- (2.5,1.2);
		\draw[dotted] (3.5,-3.2) -- (3.5,1.2);
		\draw[dotted] (4.5,-3.2) -- (4.5,1.2);
		\draw[dotted] (5.5,-3.2) -- (5.5,1.2);
		\node at (7.7,-2) {$\R^{n-1}$};
        \node[blue!60!black] at (7.1,-1) {$...$};
        \node[blue!60!black] at (7.1,-0.5) {$...$};
        \node[blue!60!black] at (7.1,-1.5) {$...$};
        \node[blue!60!black] at (-7.1,-1) {$...$};
        \node[blue!60!black] at (-7.1,-0.5) {$...$};
        \node[blue!60!black] at (-7.1,-1.5) {$...$};
\draw [fill=blue,opacity=0.25,blue] (5.25,-1.75) rectangle (6.75,-0.25);
        \draw [fill=blue,opacity=0.25,blue] (4.25,-1.75) rectangle (5.75,-0.25);
        \draw [fill=blue,opacity=0.25,blue] (3.25,-1.75) rectangle (4.75,-0.25);
        \draw [fill=blue,opacity=0.25,blue] (2.25,-1.75) rectangle (3.75,-0.25);
        \draw [fill=blue,opacity=0.25,blue] (1.25,-1.75) rectangle (2.75,-0.25);
		\draw [fill=blue,opacity=0.25,blue] (0.25,-1.75) rectangle (1.75,-0.25);
        \draw [fill=blue,opacity=0.25,blue] (-0.75,-1.75) rectangle (0.75,-0.25);
        \draw [fill=blue,opacity=0.25,blue] (-1.75,-1.75) rectangle (-0.25,-0.25);
        \draw [fill=blue,opacity=0.25,blue] (-2.75,-1.75) rectangle (-1.25,-0.25);
        \draw [fill=blue,opacity=0.25,blue] (-3.75,-1.75) rectangle (-2.25,-0.25);
        \draw [fill=blue,opacity=0.25,blue] (-4.75,-1.75) rectangle (-3.25,-0.25);
        \draw [fill=blue,opacity=0.25,blue] (-5.75,-1.75) rectangle (-4.25,-0.25);
        \draw [fill=blue,opacity=0.25,blue] (-6.75,-1.75) rectangle (-5.25,-0.25);
        \draw [fill=blue,opacity=0.25,blue!70!black,thick] (5.25,-2.75) rectangle (6.75,-1.25);
        \draw [fill=blue,opacity=0.25,blue] (4.25,-2.75) rectangle (5.75,-1.25);
        \draw [fill=blue,opacity=0.25,blue] (3.25,-2.75) rectangle (4.75,-1.25);
        \draw [fill=blue,opacity=0.25,blue] (2.25,-2.75) rectangle (3.75,-1.25);
        \draw [fill=blue,opacity=0.25,blue] (1.25,-2.75) rectangle (2.75,-1.25);
		\draw [fill=blue,opacity=0.25,blue] (0.25,-2.75) rectangle (1.75,-1.25);
        \draw [fill=blue,opacity=0.25,blue] (-0.75,-2.75) rectangle (0.75,-1.25);
        \draw [fill=blue,opacity=0.25,blue] (-1.75,-2.75) rectangle (-0.25,-1.25);
        \draw [fill=blue,opacity=0.25,blue] (-2.75,-2.75) rectangle (-1.25,-1.25);
        \draw [fill=blue,opacity=0.25,blue] (-3.75,-2.75) rectangle (-2.25,-1.25);
        \draw [fill=blue,opacity=0.25,blue] (-4.75,-2.75) rectangle (-3.25,-1.25);
        \draw [fill=blue,opacity=0.25,blue] (-5.75,-2.75) rectangle (-4.25,-1.25);
        \draw [fill=blue,opacity=0.25,blue] (-6.75,-2.75) rectangle (-5.25,-1.25);
         \node at (6,-2) {\tiny\textbullet};
         \node at (5,-2) {\tiny\textbullet};
         \node at (4,-2) {\tiny\textbullet};
         \node at (3,-2) {\tiny\textbullet};
         \node at (2,-2) {\tiny\textbullet};
        \node at (1,-2) {\tiny\textbullet};
        \node at (0,-2) {\tiny\textbullet};
        \node at (-1,-2) {\tiny\textbullet};
        \node at (-2,-2) {\tiny\textbullet};
        \node at (-3,-2) {\tiny\textbullet};
         \node at (-4,-2) {\tiny\textbullet};
          \node at (-5,-2) {\tiny\textbullet};
           \node at (-6,-2) {\tiny\textbullet};
           \draw[<->] (2,-3.5) -- (3,-3.5);
           \node at (2.5,-3.8) {$2^{-j}$};
    \draw [red] (0.25,-2.75) rectangle (1.75,-1.25);
    \node [red] at (0.5,-3.8) {$\cube_{j,1}$};
    \node [red] at (-0.5,-3.8) {$\cube'_{j,1}$};
    \draw[red] (-0.65,-3.5) -- (0.65,-1.75);
    \draw[red] (0.5,-3.5) -- (0.75,-2.75);
        \draw[white,fill=white] (-1.9,0) rectangle  (0.9,1);
		\draw [fill=red,opacity=0.1,green] (-7,-1.25) rectangle (7,1.5);
        \draw[green!10!white,fill=green!10!white] (-1.9,0) rectangle  (0.9,1);
        \node [green!50!blue] at (-0.5,0.5) {$\text{spt}(1-\eta_{j})$};
        \draw [red,dashed,thick] (0.25,-1.75) rectangle (1.75,-0.25);
	\end{tikzpicture}
\end{center}
\caption{The geometric setting in Section \ref{sec:variations}. Far away from $\R^{n-1}$, $u$ is left unchanged (see term $\mathcal{T}_{j}^{(1)}u$ in \eqref{eq:replacementsequence}) while, close to $\R^{n-1}$, it is locally replaced by a weighted sum of projections onto the finite dimensional null spaces (see term $\mathcal{T}_{j}^{(2)}u$ in \eqref{eq:replacementsequence}).}\label{fig:stratotrace1}
\end{figure}
We remark that the existence of a bounded linear trace operator $\gamma_{\R^{n-1}}\colon\homosobo^{\mathbb{A},1}(\mathbb{H})\to\lebe^{1}(\partial\mathbb{H};\R^{N})$ for arbitrary half-spaces $\mathbb{H}$ is indeed equivalent to the $\mathbb{C}$-ellipticity of $\mathbb{A}$, and the same holds true with the obvious modifications on open and bounded sets $\Omega\subset\R^{n}$ with Lipschitz boundaries, see \cite{BreitDieningGmeineder,DieningGmeineder}. 

Let us note that, in the present situation, trace operators cannot be introduced as in the gradient case, see \eqref{eq:spieler}\emph{ff}., and their construction shall be revisited below. In particular, \eqref{eq:JanetJackson} cannot be employed, and this is due to the fact that $\homosobo^{1,1}(\R_{+}^{n})\subsetneq\homosobo^{\mathbb{A},1}(\R_{+}^{n})$ in general. On $\homosobo^{1,1}(\R_{+}^{n};\R^{N})$, the trace operator on $\homosobo^{\mathbb{A},1}(\R_{+}^{n})$ coincides with the classical trace operator on $\homosobo^{1,1}(\R_{+}^{n};\R^{N})$. Since the same holds with the natural modifications for open and bounded domains with Lipschitz boundaries, the inclusion `$\subset$' in $\eqref{eq:p=1halfspaceA}_{1}$ and $\eqref{eq:p=1domainsA}_{1}$ follows directly, whereas `$\supset$' in \eqref{eq:p=1halfspaceA} and \eqref{eq:p=1domainsA} is an immediate consequence of the corresponding parts of Theorem \ref{thm:p=1}. It thus remains to establish `$\subset$' in $\eqref{eq:p=1halfspaceA}_{1}$ and $\eqref{eq:p=1domainsA}_{1}$. For this purpose, we provide a direct argument that avoids formula \eqref{eq:debramorgan}  and also yields an alternative proof for the corresponding direction in Theorem \ref{thm:p=1}.

To this end, we revisit the construction of the trace operator from \cite{BreitDieningGmeineder,DieningGmeineder}, which we give in the situation of the half-spaces $\R_{+}^{n}$; as mentioned above, the underlying argument is independent of the particular half-space. We 
describe the geometric set-up first: For $j\in\mathbb{Z}$, we consider cubes 
\begin{align}\label{eq:cubedef}
2^{-j}z + \left[-3\cdot2^{-j-2},3\cdot2^{-j-2}\right]^{n},\qquad z\in\mathbb{Z}^{n}, 
\end{align}
and enumerate these cubes by the index $i\in\mathbb{N}$.  For each $j\in\mathbb{Z}$, this gives us a sequence $(\cube_{j,i})_{i\in\mathbb{N}}$. Since the cubes obtained in this way are a scaled version of the corresponding cubes for $j=0$, they share the following properties: 
\begin{enumerate}[label=({C}{{\arabic*}})]
	\item\label{item:cubo1} For each $j\in\mathbb{Z}$, the cubes $\cube_{j,i}$ cover $\R^{n}$: 
	\begin{align*}
		\R^{n}=\bigcup_{i\in\mathbb{N}}\cube_{j,i}. 
	\end{align*}
	\item\label{item:cubo2} There exists a constant $\mathtt{N}=\mathtt{N}(n)\in\mathbb{N}$ such that for each $x\in\R^{n}$ and each $j\in\mathbb{N}$, there are at most $\mathtt{N}$ cubes $\cube_{j,i}$ such that $x\in\cube_{j,i}$. In particular, the covering $(\cube_{j,i})_{i\in\mathbb{N}}$ is uniformly locally finite. 
	\item\label{item:cubo3} There exists a constant $c>1$ with the following property: If $i_{1},i_{2}\in\mathbb{N}$ are such that $\cube_{j,i_{1}}\cap \cube_{j,i_{2}}\neq\emptyset$, then 
	\begin{align*}
	\tfrac{1}{c}2^{-jn} \leq \mathscr{L}^{n}(\cube_{j,i_{1}}\cap\cube_{j,i_{2}})\leq c 2^{-jn}. 
	\end{align*}
\end{enumerate}
Based on \ref{item:cubo1}--\ref{item:cubo3}, we may choose a partition of unity  $(\varphi_{j,i})_{i\in\mathbb{N}}$ with the following properties: 
\begin{enumerate}[label=({PU}{{\arabic*}})]
	\item\label{item:POU1} \emph{Partition of unity:} For each $j\in\mathbb{Z}$, $\varphi_{j,i} \in\hold_{c}^{\infty}(\cube_{j,i};[0,1])$, and we have that 
	\begin{align*}
		\sum_{i\in\mathbb{N}}\varphi_{j,i}\equiv 1\qquad \text{in}\;\R_{+}^{n}. 
	\end{align*}
	\item\label{item:POU2}  \emph{Scaled derivative bounds:} There exists a constant $c>0$ such that
\begin{align}\label{eq:gradbound0}
	\max_{|\alpha|=1}\sup_{\substack{i\in\mathbb{N} \\ j \in\mathbb{Z}}}	2^{-j}\|\partial^{\alpha}\varphi_{j,i}\|_{\lebe^{\infty}(\R_{+}^{n})}\leq c.
	\end{align}
\end{enumerate}
To conclude our construction, we define for $j\in\mathbb{Z}$ 
\begin{align*}
\mathbb{H}_{j}\coloneqq\{x=(x_{1},...,x_{n})\in\R_{+}^{n}\colon\;0<x_{n}<2^{-j}\}.
\end{align*}
We subsequently choose a localization function  $\eta_{j}\in\hold^{\infty}(\overline{\R_{+}^{n}};[0,1])$ such that
\begin{align}\label{eq:etaloco} \mathbbm{1}_{\mathbb{H}_{j+1}}\leq\eta_{j}\leq\mathbbm{1}_{\mathbb{H}_{j}}
\end{align}
and
\begin{align}\label{eq:gradbound}
\max_{|\alpha|= 1}\sup_{j\in\mathbb{Z}}2^{-j}\|\partial^{\alpha}\eta_{j}\|_{\lebe^{\infty}(\R_{+}^{n})}\leq c.
\end{align}
Lastly, if a cube $\cube_{j,i}$ is defined by \eqref{eq:cubedef} for some $z=(z_{1},...,z_{n})\in\mathbb{Z}^{n}$, we put 
\begin{align*}
 \cube'_{j,i} := 2^{-j}(z_{1},...,z_{n-1},z_{n}+1) + \left[-3\cdot 2^{-j-2},3\cdot 2^{-j-2}\right]^{n}. 
\end{align*}
In other words, $\cube'_{j,i}$ is obtained by shifting $\cube_{j,i}$ by $2^{-j}$ in the $n$-th coordinate direction. 

Now let $\cube\subset\R^{n}$ be a cube. As a key consequence of $\mathbb{C}$-ellipticity, the null space 
\begin{align*}
\ker(\mathbb{A};\cube)\coloneq \{u\in\lebe_{\locc}^{1}(\cube;\R^{N})\colon\;\mathbb{A}u=0\;\text{in}\;\mathscr{D}'(\cube;\R^{M})\} 
\end{align*}
is finite dimensional and is contained in the space $\mathscr{P}_{m}(\R^{n};\R^{N})$ of $\R^{N}$-valued polynomials of a fixed maximal degree $m\in\mathbb{N}_{0}$. We choose an orthonormal basis $\{\pi_{1},...,\pi_{l}\}$ of $\ker(\mathbb{A};[0,1]^{n})$ with respect to the $\lebe^{2}([0,1]^{n};\R^{N})$-inner product and define 
\begin{align}\label{eq:projectiondef}
\Pi_{[0,1]^{n}}^{\mathbb{A}}u\coloneqq \sum_{i=1}^{l}\langle u,\pi_{i}\rangle \pi_{i},\qquad u\in\lebe^{2}([0,1]^{n};\R^{N}). 
\end{align}
By translations and rotations, this gives rise to projection operators $\Pi_{\cube}^{\mathbb{A}}\colon\lebe^{2}(\cube;\R^{N})\to\ker(\mathbb{A};\cube)$. We note that \eqref{eq:projectiondef} and hereafter $\Pi_{\cube}^{\mathbb{A}}u$ are well-defined if merely $u\in\lebe^{1}(\cube;\R^{N})$; this is due to the fact that the $\pi_{i}$'s are polynomials. 
\begin{lemma}
Let $\mathbb{A}$ be a first order $\mathbb{C}$-elliptic differential operator of the form described above. Then there exist constants  $c,C>0$ such that 
\begin{align}\label{eq:inverse}
\|\Pi^{\mathbb{A}}_{\cube}u\|_{\lebe^{\infty}(\cube)}\leq c\,\dashint_{\cube}|\Pi^{\mathbb{A}}_{\cube}u|\,\dif x \leq  C\,\dashint_{\cube}|u|\,\dif x
\end{align}
holds for all $u\in\homosobo^{\mathbb{A},1}(\R_{+}^{n})$ and all cubes $\cube\Subset\R_{+}^{n}$.
\end{lemma}
The preceding lemma is taken from \cite[Lemma 3.5]{DieningGmeineder}. We note that $\eqref{eq:inverse}_{1}$ is a consequence of $\dim(\ker(\mathbb{A};\cube))\leq m<\infty$ and the equivalence of all norms on finite dimensional spaces together with the correct scaling. We are now ready to give the:

\begin{proof}[Proof of Theorem \ref{thm:operatoradapted}] First, we recall the construction of the trace operator from \cite{BreitDieningGmeineder,DieningGmeineder}. As mentioned above, the gradients of $\homosobo^{\mathbb{A},1}$-functions do not have to exist as $\lebe^{1}$-maps, whereby an approach different from the usual one (see \eqref{eq:JanetJackson}\emph{ff.}) is required. For this reason and following \cite[Section 5]{DieningGmeineder}, we define for $u\in\homosobo{^{\mathbb{A},1}}(\overline{\R_{+}^{n}})$ the approximating  sequence $(\mathcal{T}_{j}u)_{j\in\mathbb{Z}}$ by 
\begin{align}\label{eq:replacementsequence}
\mathcal{T}_{j}u \coloneqq \mathcal{T}_{j}^{(1)}u+\mathcal{T}_{j}^{(2)}u \coloneqq (1-\eta_{j})u + \eta_{j}\sum_{i\in\mathbb{N}}\varphi_{j,i}\Pi_{\cube'_{j,i}}^{\mathbb{A}}u. 
\end{align}
In essence, \eqref{eq:replacementsequence} does not change $u$ far away from $\R^{n-1}$, whereas it replaces $u$ by a weighted sum of projections onto the null space of $\mathbb{A}$ on cubes that intersect with $\R^{n-1}$; see Figure \ref{fig:stratotrace1} for this geometric setting. We note that, since $\mathcal{T}_{j}^{(2)}u$ is a locally finite sum of smooth functions, its restriction to $\R^{n-1}$ is well-defined.

In \cite[Proposition 5.2, Eqs. (5.9)--(5.13) and Remark 5.5]{DieningGmeineder}, it is shown that 
\begin{align}\label{eq:traceoperatorWA1}
\gamma_{\R^{n-1}}u \coloneqq \lim_{j\to\infty}\mathcal{T}_{j}^{(2)}u(\cdot,0) 
\end{align}
is well-defined, where the limit is taken with respect to strong convergence in $\lebe^{1}(\R^{n-1})$, and one has the estimate 
\begin{align}\label{eq:continuitytraceWA1}
\|\gamma_{\R^{n-1}}u\|_{\lebe^{1}(\R^{n-1})}\leq c\|\mathbb{A}u\|_{\lebe^{1}(\R_{+}^{n})}\qquad\text{for all}\;u\in\homosobo^{\mathbb{A},1}(\R_{+}^{n}).
\end{align}
The estimates underlying \eqref{eq:traceoperatorWA1} and \eqref{eq:continuitytraceWA1} rely on a telescope sum argument, which allows to exclusively work on the elements of the null spaces by use of Poincar\'{e}-type inequalities, see \cite[Eq. (5.9)\emph{ff.}]{DieningGmeineder}. This, in turn, allows to transfer the estimates of differences $\mathcal{T}_{j+1}^{(2)}u-\mathcal{T}_{j}^{(2)}u$ to elements of the null spaces, where all norms are equivalent and the basic obstruction from \cite{Ornstein} thus becomes invisible; in particular, on such elements, the scaled $\sobo^{1,1}$-norm on cubes is equivalent to the $\lebe^{1}$-norm on cubes. Moreover, there holds $\gamma_{\R^{n-1}}u(0,\cdot)=u|_{\R^{n-1}}$ whenever $u\in\hold(\overline{\R_{+}^{n}};\R^{N})\cap\homosobo^{\mathbb{A},1}(\R_{+}^{n})$, whereby \eqref{eq:traceoperatorWA1} is indeed the uniquely determined trace operator by \eqref{eq:continuitytraceWA1} and density. 

Because of this particular definition of the trace operator on $\homosobo^{\mathbb{A},1}(\R_{+}^{n})$, we must directly establish that $\|\mathcal{T}{_{j}^{(2)}}u\|_{\lebe^{\infty}(\R_{+}^{n})}$ can be bounded by $c\|u\|_{\lebe^{\infty}(\R_{+}^{n})}$ provided that $u\in\homosobo^{\mathbb{A},1}(\R_{+}^{n})\cap\lebe^{\infty}(\R_{+}^{n};\R^{N})$. In the following, we fix such a map $u$. Now, for $j\in\mathbb{Z}$ and $i\in\mathbb{N}$, we note that if $\cube_{j,i}$ is centered at some element of $\R^{n-1}$, then $\cube'_{j,i}\subset 3\cube_{j,i}$; here, $3\cube_{j,i}$ denotes the cube with the same center and orientation as $\cube_{j,i}$ but thrice its side length. We note that both 
\begin{align*}
v\mapsto \dashint_{3\cube_{j,i}}|v|\,\dif x\;\;\;\text{and}\;\;\;v\mapsto\dashint_{\cube'_{j,i}}|v|\,\dif x \qquad\text{are norms on $\mathscr{P}_{m}(\R^{n};\R^{N})$}. 
\end{align*}
Since all norms are equivalent on finite dimensional spaces, a continuity and scaling argument implies that there exists a constant $c=c(n,m,N)\geq 1$ such that 
\begin{align}\label{eq:dangerousMJ}
\frac{1}{c}\dashint_{3\cube_{j,i}}|v|\,\dif x \leq \dashint_{\cube'_{j,i}}|v|\,\dif x \leq c\dashint_{3\cube_{j,i}}|v|\,\dif x\qquad\text{for all}\;v\in\mathscr{P}_{m}(\R^{n};\R^{N}). 
\end{align}
Therefore, using the equivalence of all norms on the null space of $\mathbb{A}$ in conjunction with scaling in the first step, we obtain 
\begin{align}\label{eq:billiejean}
\begin{split}
\|\Pi_{\cube'_{j,i}}^{\mathbb{A}}u\|_{\lebe^{\infty}(\cube_{j,i})} & \leq c \dashint_{\cube_{j,i}}|\Pi_{\cube'_{j,i}}^{\mathbb{A}}u|\,\dif x \leq c \dashint_{3\cube_{j,i}}|\Pi_{\cube'_{j,i}}^{\mathbb{A}}u|\,\dif x \\ & \!\!\!\! \stackrel{\eqref{eq:dangerousMJ}}{\leq} c \dashint_{\cube'_{j,i}}|\Pi_{\cube'_{j,i}}^{\mathbb{A}}u|\,\dif x \stackrel{\eqref{eq:inverse}}{\leq} c \dashint_{\cube'_{j,i}}|u|\,\dif x \leq c\,\|u\|_{\lebe^{\infty}(\R_{+}^{n})}, 
\end{split}
\end{align}
and so we conclude for $x\in\overline{\R_{+}^{n}}$ that 
\begin{align}\label{eq:minuano}
\begin{split}
|\mathcal{T}_{j}^{(2)}u(x)| & \leq \sum_{i\in\mathbb{N}}\varphi_{j,i}(x)\|\Pi_{\cube'_{j,i}}^{\mathbb{A}}u\|_{\lebe^{\infty}(\cube_{j,i})} \\ & \!\!\!\! \stackrel{\eqref{eq:billiejean}}{\leq} c\,\Big(\sum_{i\in\mathbb{N}}\varphi_{j,i}(x)\Big)\|u\|_{\lebe^{\infty}(\R_{+}^{n})} = c \|u\|_{\lebe^{\infty}(\R_{+}^{n})}. 
\end{split}
\end{align}
By \eqref{eq:traceoperatorWA1}, there exists a non-relabelled subsequence such that 
\begin{align*}
\gamma_{\R^{n-1}}u(x')=\lim_{j\to\infty}\mathcal{T}_{j}^{(2)}u(x',0)\qquad\text{for $\mathscr{H}^{n-1}$-a.e. $x'\in\R^{n-1}$}, 
\end{align*}
whereby \eqref{eq:continuitytraceWA1} and \eqref{eq:minuano} immediately imply that 
\begin{align*}
\|\gamma_{\R^{n-1}}u\|_{(\lebe^{1}\cap\lebe^{\infty})(\R^{n-1})}\leq c\,\Big(\|u\|_{\lebe^{\infty}(\R_{+}^{n})}+\|\mathbb{A}u\|_{\lebe^{1}(\R_{+}^{n})}\Big)
\end{align*}
for all $u\in\homosobo^{\mathbb{A},1}(\R_{+}^{n})\cap\lebe^{\infty}(\R_{+}^{n};\R^{N})$. This completes the proof. 
\end{proof}
\begin{remark}
The construction of the trace operator via the replacement sequence as in \eqref{eq:replacementsequence}\emph{ff}. has some vague resemblance to the atomic approach to traces in Besov- or Triebel-Lizorkin spaces; see, e.g., \cite[\S 3]{Schneider}.
\end{remark}
\begin{remark}
Theorem \ref{thm:operatoradapted} immediately applies mutatis mutandis to functions of bounded $\mathbb{A}$-variation, in formulas $u\in\bv^{\mathbb{A}}(\Omega)$, where $\Omega\subset\R^{n}$ is open and bounded with Lipschitz boundary. This space is defined as the space of those $u\in\lebe^{1}(\Omega;\R^{N})$ for which $\mathbb{A}u$ is a finite $\R^{M}$-valued Radon measure on $\Omega$. 
\end{remark}
\begin{remark}\label{rem:generalizations}
The existence of a bounded linear trace operator $\gamma_{\partial\mathbb{H}}\colon \homosobo^{\mathbb{A},1}(\mathbb{H})\to\lebe^{1}(\partial\mathbb{H};\R^{N})$ for \emph{every} half-space $\mathbb{H}\subset\R^{N}$ is indeed equivalent to the $\mathbb{C}$-ellipticity of $\mathbb{A}$, see \cite{BreitDieningGmeineder,GmeinederRaitaVanSchaftingen2021,GmeinederRaitaVanSchaftingen2024}; an analogous assertion holds on open and bounded domains $\Omega\subset\R^{n}$ with Lipschitz boundaries. The key examples for the necessity of $\mathbb{C}$-ellipticity for such a trace operator rely on the elements of the null spaces of $\mathbb{A}$ which are not  $\frac{n}{n-1}$-integrable; see \cite[Theorem 4.18]{BreitDieningGmeineder} and \cite[Theorem 1.1]{GmeinederRaita19}. It is then natural to inquire as to whether a slight weakening of the underlying ellipticity might imply the existence of a trace operator mapping to $\lebe^{1}(\partial\Omega;\R^{N})$ \emph{if} an a priori higher integrability assumption is made, potentially ruling out the critical singularities of elements of the null spaces. It shall be established in the follow-up paper \cite{ChenEtAl} that this is not possible either, and so the general setting of Theorem \ref{thm:operatoradapted} is optimal indeed.
\end{remark}

\section*{Acknowledgments} 
{\small F.G. is grateful to the Hector Foundation for financial support, Project Number FP 626/21. P.S. acknowledges financial support of his PhD studies at the University of Konstanz by a scholarship of the LGFG (Landesgraduiertenf\"{o}rderungsgesetz) Baden-W\"{u}rttemberg.}

\bibliographystyle{alpha}
\bibliography{DGS}

\end{document}